\documentclass[11pt,oneside,english]{amsart}
\usepackage[english]{babel}
\usepackage{graphics}
\usepackage{amsfonts}
\usepackage{amsmath}
\usepackage{amssymb}
\usepackage{mathrsfs}
\usepackage{enumerate}
\usepackage{epsfig}
\usepackage{amscd}
\usepackage{graphicx}
\usepackage{newlfont}
\usepackage{latexsym}
\usepackage{amsthm}
\usepackage{color}
\usepackage{hyperref}
\hypersetup{colorlinks=true,linkcolor=blue,citecolor=blue}
\usepackage{pstricks, pst-node, pst-text, pst-3d}
\definecolor{hellgrau}{gray}{.9}

\makeatletter \theoremstyle{plain}
\newtheorem{theorem}{Theorem}[section]
\newtheorem{lemma}[theorem]{Lemma}
\newtheorem{corollary}[theorem]{Corollary}
\newtheorem{proposition}[theorem]{Proposition}

\theoremstyle{remark}
\newtheorem{remark}[theorem]{Remark}
\newtheorem{question}[theorem]{Question}
\numberwithin{equation}{section}

\theoremstyle{definition}
\newtheorem{definition}[theorem]{Definition}
\newtheorem{example}[theorem]{Example}

\newcommand{\ep}{\epsilon}
\newcommand{\del}{\delta}

\newcommand{\reals}{\mathbb{R}}

\newcommand{\cH}{{\mathcal{H}}}
\newcommand{\cK}{{\mathcal{K}}}
\newcommand{\cP}{{\mathcal{P}}}

\newcommand{\boldu}{{\mathbf{u}}}
\newcommand{\boldv}{{\mathbf{v}}}

\newcommand{\ovl}{\overline}

\newcommand{\he}{\textnormal{H}}

\def\divh{\operatorname{div_{\textnormal{H}}}}
\def\trace{\operatorname{tr}}

\newcommand{\CoSi}{\nabla^{\Sigma}}
\newcommand{\cu}{\dot{\gamma}}
\newcommand{\cuo}{\dot{\gamma}_{1}}
\newcommand{\cut}{\dot{\gamma}_{2}}

\newcommand{\cf}{\omega(\dot{\gamma})}

\newcommand{\grad}{\nabla}
\newcommand{\hgrad}{\nabla_{\textnormal{H}}}

\newcommand{\hlapl}{\Delta_{\textnormal{H}}}
\newcommand{\eps}{\varepsilon}

\newcommand{\leb}{\mathcal{L}}

\newcommand{\bound}{\partial \Omega}

\newcommand{\heis}{\mathbb{H}}
\newcommand{\vef}{\mathcal{X}(M)}

\newcommand{\scal}[2]{\langle {#1} , {#2}\rangle}
\newcommand{\Mn}{\mathcal{N}}
\newcommand{\XtL}{X_{3}^{L}}

\begin{document}

\title{Sub-Riemannian curvature and a Gauss--Bonnet theorem in the Heisenberg group}
\author{Zolt\'an Balogh}
\address{Mathematisches Institut, Universit\"at Bern, Sidlerstrasse 5, 3012 Bern, Switzerland}
\email{zoltan.balogh@math.unibe.ch}

\author{Jeremy T. Tyson}
\address{Department of Mathematics \\ University of Illinois \\ 1409 West Green St. \\ Urbana, IL, 61801}
\email{tyson@illinois.edu}

\author{Eugenio Vecchi}
\address{Dipartimento di Matematica, Universit\`{a}  di Bologna, Piazza di Porta San Donato 5, 40126 Bologna, Italy}
\email{eugenio.vecchi2@unibo.it}

\date{\today}
\thanks{ZMB and EV were supported by the Swiss National Science Foundation Grant No.\ 200020-146477, and have also received funding from the People Programme (Marie Curie Actions) of the European Union's Seventh
Framework Programme FP7/2007-2013/ under REA grant agreement No.\ 607643 (ERC Grant MaNET `Metric Analysis for Emergent Technologies').
JTT acknowledges support from U.S. National Science Foundation Grant DMS-0120870 and Simons Foundation Collaboration Grant 353627.}
\keywords{Heisenberg group, sub-Riemannian geometry, Riemannian approximation, Gauss--Bonnet theorem, Steiner formula}
\subjclass[2010]{Primary 53C17; Secondary 53A35, 52A39}

\maketitle
 \begin{abstract}
We use a Riemannnian approximation scheme to define a notion of \textit{sub-Riemannian Gaussian curvature}
for a Euclidean $C^{2}$-smooth surface in the Heisenberg group $\heis$
away from characteristic points, and a notion of \textit{sub-Riemannian signed geodesic curvature}
for Euclidean $C^{2}$-smooth curves on surfaces. These results
are then used to prove a Heisenberg version of the Gauss--Bonnet theorem. An application to Steiner's formula for the Carnot-Carath\'eodory distance in $\heis$ is provided.
\end{abstract}

\tableofcontents

\section{Introduction}
A full understanding of the notion of curvature has been at the
core of studies in differential geometry
since the foundational works
of Gauss and Riemann.
The aim of this paper is to propose a suitable candidate for the
notion of \textit{sub-Riemannian Gaussian curvature} for Euclidean
$C^{2}$-smooth surfaces in the first Heisenberg group $\heis$,
adopting the so called \textit{Riemannian approximation scheme},
which has proved to be a very powerful tool to address sub-Riemannian issues.

Referencing the seminal work of Gauss, 
we recall that to a compact and oriented
Euclidean $C^{2}$-smooth regular surface
$\Sigma \subset \reals^{3}$ we can attach the notions of
\textit{mean curvature} and \textit{Gaussian curvature} as
symmetric polynomials of the second fundamental form.
To be more precise, for every $p \in \Sigma$ we have
a well-defined \textit{outward unit normal} vector field,
$N(p):\Sigma \to \mathbb{S}^{2}$,
usually called the \textit{Gauss normal map}.
For every $p \in \Sigma$, the differential
of the Gauss normal map $dN(p) : T_{p}\Sigma \to T_{N(p)}\mathbb{S}^{2},$
defines a positive definite and symmetric
quadratic form on $T_{p}\Sigma$ whose two real eigenvalues are usually
called \textit{principal curvatures} of $\Sigma$ at $p$.
The arithmetic mean of these principal curvatures is the \textit{mean curvature}
and their product is the \textit{Gaussian curvature}.
The importance of the latter became particularly clear after
Gauss' famous \textit{Theorema Egregium}, which asserts that
Gaussian curvature is intrinsic and is also an
isometric invariant of the surface $\Sigma$.

The notions of curvature, as briefly recalled above,
can be extended to far more general situations, for instance to submanifolds of higher codimension in $\reals^{n}$,
and also to the broader geometrical context provided by Riemannian geometry,
as was done by Riemann.
In particular, we will consider 2-dimensional Riemannian manifolds isometrically embedded
into 3-dimensional Riemannian manifolds. We refer to Section \ref{C3} for details.

Our interest in the study of curvatures of
surfaces in $\heis$ is motivated by the still ongoing studies in the context of
sub-Riemannian manifolds or more specific structures like Carnot groups,
whose easiest example is provided by the first Heisenberg group $\heis$.
Restricting our attention to $\heis$,
there is a currently accepted notion of \textit{horizontal mean curvature} $\cH_0$
at non-characteristic points of Euclidean regular surfaces.
This notion has been considered by Pauls (\cite{P04}) via the method of
Riemannian approximants, but has also been proved to be equivalent to
other notions of mean curvature appearing in the literature (e.g.\ \cite{CHMY05} or \cite{DGN07}).

The method of Riemannian approximants relies on a famous result due to Gromov, which states that the metric
space $(\mathbb{H},d_{cc})$ can be obtained as
the pointed Gromov-Hausdorff limit of
a family of metric spaces $(\mathbb{R}^{3}, g_L)$,
where $g_L$ is a suitable family
of Riemannian metrics. The Riemannian approximation scheme has also proved
to be a very efficient tool in more analytical settings, for instance, in the study
of estimates for fundamental solutions of the sub-Laplacian $\Delta_{\he}$ (e.g.\ \cite{CM06, CCM07}) as well as regularity theory for sub-Riemannian curvature flows (e.g.\ \cite{ccm:mcf}). The preceding represents only a small sample of the many applications of the Riemannian approximation method in sub-Riemannian geometric analysis, and we refer the reader to the previously cited papers for more information and references to other work in the literature. The monograph \cite{CDPT} provides a detailed description of the Riemannian approximation scheme in the setting of the Heisenberg group.

Let us denote by $X_{1}, X_{2}$ and $X_{3}$ the left-invariant vector
fields which span the Lie algebra $\mathfrak{h}$ of $\heis$. In particular,
$[X_1, X_2]=X_3$. In order to exploit the contact nature of $\heis$ it
is customary to define an inner product $\scal{\cdot}{\cdot}_{\he}$ which makes
$\{X_1, X_2\}$ an orthonormal basis. A possible way to define a Riemannian scalar
product is to set $\XtL := X_{3}/\sqrt{L}$ for every $L>0$,
and then to extend $\scal{\cdot}{\cdot}_{\he}$ to a scalar product $\scal{\cdot}{\cdot}_{L}$ which makes
$\{X_1,X_2, \XtL\}$ an orthonormal basis.
The family of metric spaces $(\reals^{3}, g_{L})$ converges to $(\heis, d_{cc})$ in the pointed Gromov-Hausdorff sense.

Within this family of Riemannian manifolds, we can now perform computations
adopting the unique Levi-Civita connection associated to the family of Riemannian metrics
$g_{L}$. Obviously, all the results are expected depend on the positive constant $L$.
The plan is to extract horizontal notions out of the computed objects and to study
their asymptotics in $L$ as $L \to +\infty$. This is the technique adopted
in \cite{P04} to define a notion of horizontal mean curvature.

It is natural to ask whether such a method can be employed
to study the curvature of curves, and especially to articulate an appropriate
notion of \textit{sub-Riemannian Gaussian curvature}.
One attempt in this direction has been carried out in \cite{CPT10}, where the authors proposed
a notion of horizontal second fundamental form in relation with $H$-convexity.
A different notion of sub-Riemannian Gaussian curvature for graphs has been suggested in \cite{DGN03}.

Our approach follows closely the classical theory of Riemannian geometry and leads us to
the following notion of \textit{sub-Riemannian curvature} for
a Euclidean $C^{2}$-smooth and regular curve $\gamma=(\gamma_1, \gamma_2, \gamma_3):[a,b]\to \heis$:
\begin{equation}\label{eq:k0INTRO}
k_{\gamma}^{0} =
\begin{cases}
\dfrac{|\cuo \ddot{\gamma}_2 - \cut \ddot{\gamma}_1|}{(\cuo^2 + \cut^2)^{3/2}}, &\mbox{if $\gamma(t)$ is a horizontal point of $\gamma$,}  \\

\\
\dfrac{\sqrt{\cuo^{2}+ \cut^{2}}}{|\omega(\dot{\gamma})|}, & \mbox{if $\gamma(t)$ is not a horizontal point of $\gamma$.}
\end{cases}
\end{equation}
Here $\omega = dx_3 -\tfrac{1}{2}\left( x_1 dx_2 - x_2 dx_1 \right)$ is the
standard contact form on $\reals^3$. We stress that, when dealing with \textit{purely} horizontal curves, the above notion
of curvature is already known and appears frequently in the literature.

An analogous procedure allow us to define also a notion
of \textit{sub-Riemannian signed geodesic curvature} for
Euclidean $C^{2}$-smooth and regular curves $\gamma=(\gamma_1, \gamma_2, \gamma_3):[a,b]\to \Sigma \subset \heis$
living on a surface $\Sigma = \{ x\in \heis: u(x)=0\}$, with
$u \in C^{2}(\reals^{3})$. This notion takes the form
\begin{equation}\label{eq:ksIntro}
k_{\gamma, \Sigma}^{0,s} =
\left \{
\begin{array}{rl}
\dfrac{\bar{p} \cuo + \bar{q} \cut}{|\omega(\cu)|}, & \mbox{if $\gamma(t)$ is a non-horizontal point,} \\
0 , & \mbox{if $\gamma(t)$ is a horizontal point,}
\end{array}\right.
\end{equation}
where $\hgrad u = (X_1u,X_2u)$, $\bar{p}=\tfrac{X_1 u}{\|\hgrad u \|_{\he}}$ and $\bar{q}=\tfrac{X_2 u}{\|\hgrad u \|_{\he}}$. We refer to Section \ref{C1} and Section \ref{C2} for precise statements and definitions.

In the same spirit we introduce a notion
of \textit{sub-Riemannian Gaussian curvature} $\cK_0$
away from characteristic points.
We will work with
Euclidean $C^{2}$-smooth surfaces $\Sigma = \{x\in \heis:
u(x)=0\}$, whose characteristic set
$C(\Sigma)$ is defined as the set of points $x\in \Sigma$
where $\hgrad u(x)=(0,0)$.
The explicit expression of $\cK_0$ reads as follows:
\begin{equation}\label{eq:K0intro}
\cK_0 = -\left( \dfrac{X_3 u}{\|\hgrad u\|_{\he}}\right)^{2}
-  \left( \dfrac{X_2 u}{\|\hgrad u\|_{\he}}\right) X_1 \left(\dfrac{X_3 u}{\|\hgrad u\|_{\he}}\right)
+ \left( \dfrac{X_1 u}{\|\hgrad u\|_{\he}} \right) X_2 \left(\dfrac{X_3 u}{\|\hgrad u\|_{\he}} \right).
\end{equation}
The quantity in \eqref{eq:K0intro} cannot easily be viewed as a symmetric polynomial of any kind of
horizontal Hessian. Moreover, the expression of $\cK_0$ written
above resembles one of the integrands, the one which would be
expected to replace the classical Gaussian curvature,
appearing in the Heisenberg Steiner's formula proved in \cite{BFFVW}.
The discrepancy between these two quantities will be the object
of further investigation.

The definition of an appropriate notion of sub-Riemannian
Gaussian curvature leads to
the question of proving a suitable Heisenberg
version of the celebrated Gauss--Bonnet Theorem,
which is the first main result of this paper.
For a surface $\Sigma = \{ x \in \heis: u(x)=0\},$
with $u \in C^{2}(\reals^{3})$, our main theorem is as follows.

\begin{theorem}\label{HGB}
Let $\Sigma \subset \heis$ be a regular surface with finitely many
boundary components $(\partial \Sigma)_{i}$, $i \in \{1, \ldots,n\}$, given by
Euclidean $C^{2}$-smooth regular and closed curves $\gamma_{i}:[0,2\pi] \to (\partial \Sigma)_{i}.$
Let $\cK_0$ be the sub-Riemannian Gaussian curvature of $\Sigma$, and $k_{\gamma_{i}, \Sigma}^{0,s}$
the sub-Riemannian signed geodesic curvature of $\gamma_{i}$ relative to $\Sigma$.
Suppose that the characteristic set $C(\Sigma)$ satisfies $\mathcal{H}_{E}^{1}(C(\Sigma))=0$,
and that $\|\hgrad u \|_{\he}^{-1}$ is locally summable with respect to
the Euclidean 2-dimensional Hausdorff measure near the characteristic set $C(\Sigma)$.
Then
$$
\int_{\Sigma}\cK_0 \, d\mathcal{H}_{cc}^{3}
+ \sum_{i=1}^{n}\int_{\gamma_{i}}k_{\gamma_{i}, \Sigma}^{0,s} \, d\dot{\gamma}_{i}
=0.
$$
\end{theorem}

The sharpness of the assumption made on the 1-dimensional
Euclidean Hausdorff measure $\mathcal{H}_{E}^{1}(C(\Sigma))$
of the characteristic set $C(\Sigma)$ is discussed in Section \ref{gb},
while comments on the local summability asked for $\|\hgrad u\|_{\he}^{-1}$
are postponed to Section \ref{questions}. The measure $d\dot{\gamma}_i$ on the $i$th boundary curve $(\partial\Sigma)_i$
in the statement of Theorem \ref{HGB} is the limit of scaled length measures in the Riemannian approximants. We remark that this measure vanishes along purely horizontal boundary curves.

Gauss--Bonnet type theorems have previously been obtained by Diniz and Veloso \cite{DV12} for non-characteristic surfaces in $\heis$,
and by Agrachev, Boscain and Sigalotti \cite{ABS08} for almost-Riemannian structures. We would also like to mention the results obtained by Bao and Chern \cite{BC96} in Finsler spaces.

The notion of \textit{horizontal mean curvature} has featured
in a long and ongoing research program concerning
the study of \textit{constant mean curvature} surfaces in
$\heis$, especially in relation to {\it Pansu's isoperimetric problem} (e.g. \cite{N04}, \cite{RR06}, \cite{HP08}, \cite{dgn:partial} or
\cite{CDPT}).
A simplified version of the aforementioned Gauss--Bonnet Theorem \ref{HGB}, i.e.,
when we consider a compact, oriented, Euclidean $C^{2}$-surface with no boundary,
or with boundary consisting of fully horizontal curves, ensures that the only compact
surfaces with constant \textit{sub-Riemannian Gaussian curvature} have $\cK_0 = 0$.

Our main application concerns a Steiner's formula for non-characteristic surfaces. This result
(see Theorem \ref{Torus}) is a simplification of the Steiner's formula recently proved in \cite{BFFVW}.

The structure of the paper is as follows. In Section \ref{not} we provide
a short introduction to the first Heisenberg group $\heis$ and the
notation which we will use throughout the paper, with a special focus to the Riemannian approximation scheme.
In Section \ref{C1} and \ref{C2} we adopt the Riemannian
approximation scheme to derive the expression (\ref{eq:k0INTRO}) for the
sub-Riemannian curvature of Euclidean $C^{2}$-smooth curves in $\heis$,
and the expression (\ref{eq:ksIntro}) for the sub-Riemannian geodesic curvature of curves
on surfaces. In Section \ref{C3}, we will derive the expression
(\ref{eq:K0intro}) for the \textit{sub-Riemannian Gaussian curvature}.
In Section \ref{gb} we prove Theorem \ref{HGB}
and its corollaries. Section \ref{examples} contains the proof
of Steiner's formula for non-characteristic surfaces.
In Section \ref{questions} we present a Fenchel-type
theorem for horizontal closed curves (see Theorem \ref{Fenchel})
and we pose some questions.
One of the more interesting and challenging questions concerns the summability
of the sub-Riemannian Gaussian curvature $\cK_0$
with respect to the Heisenberg perimeter measure
near isolated characteristic points. This summability issue is closely related to the open problem
posed in \cite{DGN12} concerning the summability
of the horizontal mean curvature $\cH_0$
with respect to the Riemannian surface measure near the characteristic set.
To end the paper, we add an appendix where we collect several examples
of surfaces in which we compute explicitly the sub-Riemannian Gaussian curvature $\cK_0$.

\

\paragraph{\bf Acknowlegements.} Research for this paper was conducted during visits of the second and third authors to the University of Bern in 2015 and 2016. The hospitality of the Institute of Mathematics of the University of Bern is gratefully acknowledged. The authors would also like to thank Luca Capogna for many valuable conversations on these topics and for helpful remarks concerning the proof of Theorem \ref{HGB}.

\section{Notation and background}\label{not}
Let $\heis$ be the first Heisenberg group where
the non-commutative group law is given by
$$(y_1,y_2,y_3) \ast (x_1,x_2,x_3) =
\left(x_1 + y_1, x_2 + y_2, x_3 + y_3 - \dfrac{1}{2}(x_1 y_2 - x_2 y_1) \right).$$
The corresponding Lie algebra of left-invariant vector fields
admits a $2$-step stratification,
$\mathfrak{h} = \mathfrak{v}_1 \oplus \mathfrak{v}_2$,
where $\mathfrak{v}_{1} = \mathrm{span}\{X_1,X_2\}$ and
$\mathfrak{v}_2 = \mathrm{span}\{X_3\}$ for
$X_1 = \partial_{x_1} - \tfrac12{x_2}\partial_{x_3}$,
$X_2 = \partial_{x_2} + \tfrac12{x_1}\partial_{x_3}$
and $X_3 = [X,Y] = \partial_{x_3}$.
On $\heis$ we consider also the standard contact form of $\mathbb{R}^{3}$
$$\omega = dx_3 - \dfrac{1}{2} \left( x_1 dx_2 -  x_2 dx_1 \right).$$

The left-invariant vector fields $X_{1}$ and $X_{2}$ play a major role  in the theory of
the Heisenberg group because they span a two-dimensional plane distribution $H \heis$, known as the \textit{horizontal distribution}, which is also the kernel of the contact form $\omega$:
$$
H_{x}\heis := \mathrm{span}\{ X_{1}(x), X_{2}(x)\} = (\mathrm{Ker} \omega)(x), \qquad x \in \heis.
$$
This smooth distribution of planes is a subbundle of the tangent bundle of $\heis$, and
it is a non integrable distribution because $[X_{1}, X_{2}] = X_{3} \notin H\heis$.
We can define an inner product $\scal{\cdot}{\cdot}_{x,\he}$ on $H \heis$,
so that for every $x \in \heis$, $\{X_{1}(x), X_{2}(x)\}$
forms a orthonormal basis of $H_{x} \heis$.
We will then denote by $\| \cdot \|_{x,\he}$ the horizontal norm induced
by the scalar product $\scal{\cdot}{\cdot}_{x, \he}$.
In both cases, we will omit the dependence on the base point $x \in \heis$
when it is clear.

\begin{definition}
An absolutely continuous curve $\gamma:[a,b]\subset \mathbb{R} \to \heis$ is said to
be horizontal if $\dot{\gamma}(t) \in H_{\gamma(t)}\heis$ for a.e.\ $t \in [a,b]$.
\end{definition}

\begin{definition}\label{length}
Let $\gamma:[a,b]\to \heis$ a horizontal curve.
The horizontal length $l_{\textnormal{H}}(\gamma)$ of $\gamma$ is defined as
$$l_{\textnormal{H}}(\gamma) := \int_{a}^{b} \|\cu\|_{\textnormal{H}} \, dt.$$
\end{definition}

It is standard to equip the Heisenberg group $\heis$ with a path-metric
known as Carnot-Carath\'{e}odory, or $cc$, distance:

\begin{definition}
Let $x,y \in \heis$, with $x \neq y$. The $cc$ distance between $x,y$ is defined as
$$d_{cc}(x,y):= \inf \{ l_{\textnormal{H}}(\gamma) | \gamma:[a,b] \to \heis, \gamma(a)=x, \gamma(b)=y\}$$
\end{definition}

Dilations of the Heisenberg group are defined as follows:
\begin{equation}\label{dilations}
\delta_r(x_1,x_2,x_3) = (rx_1,rx_2,r^2x_3), \qquad r>0
\end{equation}
It is easy to verify that dilations are compatible with the group operation: $\delta_r(y*x)=\delta_r(y)*\delta_r(x)$, $x,y\in\heis$, $r>0$, and that the $cc$ distance is homogeneous of order one with respect to dilations: $d_{cc}(\delta_r(x),\delta_r(y)) = r\,d_cc(x,y)$. The scaling behavior of the left-invariant vector fields $X_1,X_2,X_3$ with respect to dilations is as follows:
$$
X_1(f\circ\delta_r) = r \, X_1f \circ \delta_r, \quad
X_2(f\circ\delta_r) = r \, X_2f \circ \delta_r, \quad
X_3(f\circ\delta_r) = r^2 \, X_3f \circ \delta_r.
$$

We are now ready to implement the Riemannian approximation scheme.
First, let us define $\XtL := \tfrac{X_3}{\sqrt{L}}$ for $L>0$.
We define a family of Riemannian metrics $(g_L)_{L>0}$ on $\mathbb{R}^{3}$
such that $\{X_1,X_2,\XtL \}$ becomes an orthonormal basis.
The choice of this specific family of Riemannian metrics
on $\reals^{3}$ is indicated by the following theorem.

\begin{theorem}[Gromov]
 The family of metric spaces $(\reals^{3}, g_L)$ converges
to $(\heis, d_{cc})$ in the
 pointed Gromov-Hausdorff sense as $L \to +\infty$.
\end{theorem}

This deep result continue to hold even for more general Carnot groups,
but there is one additional feature which is valid for $\heis$:

\begin{proposition}
 Any length minimizing horizontal curve $\gamma$ joining $x \in \heis$ to
 the origin $0 \in \heis$ is the uniform limit as $L \to +\infty$ of geodesic
 arcs joining $x$ to $0$ in the Riemannian manifold $(\reals^{3},g_L)$.
\end{proposition}

For both results, we refer to \cite[Chapter 2]{CDPT}.

Continuing with notation, the scalar product that makes $\{X_1,X_2,\XtL \}$
an orthonormal basis will
be denoted by $\scal{\cdot}{\cdot}_{L}$. Explicitly,
this means that, given $V = v_1 X_1 + v_2 X_2 + v_3 X_3$
and $W = w_1 X_1 + w_2 X_2 + w_3 X_3$,
$$\scal{V}{W}_{L} = v_1 w_1 + v_2 w_2 + L v_3 w_3.$$
Obviously, if we write $V$ and $W$ in the $\{X_1,X_2,\XtL \}$ basis,
i.e., $V = v_1 X_1 + v_2 X_2 + v^{L}_{3} \XtL$
where $v^{L}_{3} = v_3 \sqrt{L}$ (and similarly for $W$), we have
$$
\scal{V}{W}_{L} =  v_1 w_1 + v_2 w_2 + v^{L}_{3} w^{L}_{3} = v_1 w_1 + v_2 w_2 + L v_3 w_3.
$$
The following relations allow us to switch from the standard basis $\{e_1 ,e_2, e_3\}$
to $\{X_1,X_2,\XtL \}$ and vice versa:
\begin{equation*}
\begin{cases}
e_1 = X_1 + \tfrac12{x_2}\sqrt{L} \XtL, &   \\
e_2 = X_2 - \tfrac12{x_1}\sqrt{L} \XtL, &   \\
e_3 = \sqrt{L} \XtL, & \\
\end{cases}
\quad \textrm{and} \quad
\begin{cases}
X_1 = e_1 - \tfrac12{x_2} e_3, &   \\
X_2 = e_2 + \tfrac12{x_1} e_3, &   \\
\XtL = \tfrac{e_3}{\sqrt{L}}.  & \\
\end{cases}
\end{equation*}
In exponential coordinates, the metric $g_L$ is represented by
the $3\times 3$ symmetric matrix $(g_{L})_{ij} := \scal{e_i}{e_j}_L$, for $i,j=1,2,3.$
In particular,
$$g_{L}(x_1,x_2,x_3)=
\begin{pmatrix}
1+ \tfrac14 {x_{2}^{2}} L & - \tfrac14 {x_{1} x_{2}} L & \tfrac12 {x_2} L \\
- \tfrac14 {x_1 x_2} L & 1+ \tfrac14 {x_{1}^{2}} L & -\tfrac12 {x_1} L \\
\tfrac12 {x_2} L & -\tfrac12 {x_1} L & L \\
\end{pmatrix}.
$$
Then $\det(g_L(x)) = L$ and
\begin{displaymath}
g_{L}^{-1}(x_1,x_2,x_3)=
\begin{pmatrix}
1 & 0 & -\tfrac{x_2}{2} \\
0 & 1 & \tfrac{x_1}{2} \\
-\tfrac12 {x_2} & \tfrac12 {x_1} & \tfrac{4+L(x_{1}^{2} + x_{2}^{2})}{4L} \\
\end{pmatrix}
\end{displaymath}
Following the classical notation of Riemannian geometry, we
will denote by $g_{ij}$ the elements of the matrix $g_L$, and by
$g^{ij}$ the elements of its inverse $g_{L}^{-1}$.\\

A standard computational tool in Riemannian geometry
is the notion of \textit{affine connection}.

\begin{definition}
Let $\vef$ be the set of $C^{\infty}$-smooth vector fields on
a Riemannian manifold $M$. Let $\mathcal{D}(M)$ be the ring
of real-valued $C^{\infty}$-smooth functions on $M$.
An affine connection $\nabla$ on $M$ is a mapping
$$\nabla : \vef \times \vef \to \vef,$$
usually denoted by $(X,Y) \mapsto \nabla_{X}Y$, such that:
\begin{itemize}
	\item [i)] $\nabla_{fX+gY}Z = f \nabla_{X}Z+ g \nabla_{Y}Z.$
	\item [ii)] $\nabla_{X}(Y+Z) = \nabla_{X}(Y) + \nabla_{X}(Z).$
	\item [iii)] $\nabla_{X}(fY) = f \nabla_{X}Y + X(f)Y,$
\end{itemize}
 for every $X,Y,Z \in \vef$ and for every $f,g \in \mathcal{D}(M)$.
\end{definition}

It is well known that every Riemannian manifold is equipped with a privileged affine connection: the Levi-Civita connection $\nabla$.
This is the unique affine connection which is compatible
with the given Riemannian metric and symmetric, i.e.,
$$
X \scal{Y}{Z}_{L} = \scal{\nabla_{X}Y}{Z}_{L} + \scal{Y}{\nabla_{X}Z}_{L}
$$
and
$$
\nabla_{X}Y - \nabla_{Y}X = [X,Y]
$$
for every $X,Y,Z \in \vef$. A direct
proof of this fact yields the famous
\textit{Koszul identity}:
\begin{equation}\label{eq:Koszul}
\begin{aligned}
\scal{Z}{\nabla_{X}Y}_{L} &= \dfrac{1}{2} \big( X \scal{Y}{Z}_{L} + Y \scal{Z}{X}_{L} -Z \scal{X}{Y}_{L} \\
&-\scal{[X,Z]}{Y}_{L} - \scal{[Y,Z]}{X}_{L} - \scal{[X,Y]}{Z}_{L}\big)
\end{aligned}\end{equation}
for $X,Y,Z \in \vef$.

It is possible to write the Levi-Civita connection $\nabla$
in a local frame by making use of the Christoffel symbols $\Gamma_{ij}^{m}$.
In our case, due to the specific nature of the Riemannian manifold $(\reals^{3},g_L)$,
we can use a global chart given by the identity map of $\reals^{3}$.
The Christoffel symbols are uniquely determined by
$$\nabla_{e_i}e_j = \Gamma_{ij}^{m}e_m, \quad i,j,m=1,2,3.$$

\begin{lemma}
The Christoffel symbols $\Gamma_{ij}^{m}$ of the Levi-Civita
connection $\nabla$ of $(\reals^{3},g_L)$ are given by
\begin{equation}\label{eq:G1}
\Gamma_{ij}^{1} =
\begin{cases}
0, & (i,j)\in \{ (1,1), (1,3), (3,1), (3,3)\}, \\
\tfrac14 {x_2} L, & (i,j)\in \{(1,2),(2,1)\}, \\
-\tfrac12 {x_1} L, & (i,j)= (2,2), \\
\tfrac12 {L}, & (i,j)\in \{(2,3),(3,2)\},
\end{cases}
\end{equation}

\begin{equation}\label{eq:G2}
\Gamma_{ij}^{2} =
\begin{cases}
-\tfrac12 {x_2} L, & (i,j)= (1,1), \\
\tfrac14 {x_1} L, & (i,j)\in \{(1,2),(2,1)\}, \\
-\tfrac12 {L}, & (i,j)\in \{(1,3),(3,1)\}, \\
0, & (i,j)\in \{ (2,2), (2,3), (3,2), (3,3)\}, \\
\end{cases}
\end{equation}
and
\begin{equation}\label{eq:G3}
\Gamma_{ij}^{3} =
\begin{cases}
-\tfrac14 {x_1 x_2} L, & (i,j)=(1,1), \\
\tfrac18 ({x_{1}^{2}-x_{2}^{2}}) L, & (i,j)\in \{(1,2),(2,1)\}, \\
-\tfrac14 {x_1} L, & (i,j)\in \{(1,3),(3,1)\}, \\
\tfrac14 {x_1 x_2} L, & (i,j)=(2,2),\\
-\tfrac14 {x_2} L, & (i,j) \in \{ (2,3), (3,2)\},\\
0, & (i,j)=(3,3). \\
\end{cases}
\end{equation}
\end{lemma}

\begin{proof}
It is a direct computation using
$$\Gamma_{ij}^{m} = \dfrac{1}{2}\sum_{k=1}^{3} \left\{ \dfrac{\partial}{\partial_{x_i}}g_{jk}
+ \dfrac{\partial}{\partial_{x_j}}g_{ki} -\dfrac{\partial}{\partial_{x_k}}g_{ij} \right\}g^{km},$$
for $i,j,m=1,2,3$.
\end{proof}

We now compute the Levi-Civita connection $\nabla$ associated to the Riemannian metric $g_L$.

\begin{lemma}
The action of the Levi-Civita connection $\nabla$ of $(\reals^{3}, g_L)$ on the vectors $X_1$, $X_2$
and $\XtL$ is given by
\begin{equation*}
\begin{aligned}
\nabla_{X_1}X_2 &= -\nabla_{X_2}X_1 = \frac12 {X_3},  \\
\nabla_{X_1}\XtL &= \nabla_{\XtL}X_1 = -\frac12 {\sqrt{L}} \, X_2,  \\
\nabla_{X_2}\XtL &= \nabla_{\XtL}X_2 = \frac12 {\sqrt{L}} \, X_1.
\end{aligned}
\end{equation*}
\end{lemma}

\begin{proof}
It follows from a direct application of the Koszul identity \eqref{eq:Koszul}, which here simplifies to
$$
\scal{Z}{\nabla_{X}Y}_{L} = -\dfrac{1}{2} \biggl( \scal{[X,Z]}{Y}_{L} + \scal{[Y,Z]}{X}_{L} + \scal{[X,Y]}{Z}_{L}\biggr).
$$
\end{proof}

To make the paper self-contained, we recall here the definitions of Riemann curvature tensor $R$ and of sectional curvature.

\begin{definition}
The Riemann curvature tensor $R$ of a
Riemannian manifold $M$ is a mapping
$R(X,Y): \vef \to \vef$ defined as follows
$$R(X,Y)Z := \nabla_{Y}\nabla_{X}Z - \nabla_{X}\nabla_{Y}Z + \nabla_{[X,Y]}Z, \quad Z \in \vef.$$
\end{definition}

\begin{remark}\label{Rfunct}
 Note that the Riemann curvature tensor $R$ satisfies the functional property
 \begin{equation}\label{eq:functional}
R(fX,Y)Z = R(X,fY)Z = R(X,Y)(fZ) = f R(X,Y)Z,
 \end{equation}
 for every $X,Y,Z \in \vef$ and every $f \in \mathcal{D}(M)$.
\end{remark}

\begin{definition}\label{sectional}
Let $M$ be a Riemannian manifold and let
$\Pi \subset T_{p}M$ be a two-dimensional subspace of the tangent
space $T_{p}M$. Let $\{E_{1}, E_{2}\}$ be two linearly independent vectors in $\Pi$.
The sectional curvature $K(E_{1},E_{2})$ of $M$ is defined as
$$K(E_{1}, E_{2}) := \dfrac{\scal{R(E_{1},E_{2})E_{1}}{E_{2}}}{|E_{1}\wedge E_{2}|^{2}},$$
 where $\wedge$ denotes the usual wedge product.
\end{definition}

One of the main reasons to introduce the notion of affine connection,
is to be able to differentiate smooth vector fields along curves:
this operation is known as \textit{covariant differentiation} (see \cite{dC}, Chapter 2).
Formally, let $Z = z_1 e_1 + z_2 e_2 + z_3 e_3$  be a smooth vector field
(written in the standard basis of $\mathbb{R}^{3}$)
along a curve $\gamma = \gamma(t)$.
The covariant derivative of $Z$ along the curve $\gamma$
is given by
\begin{equation*}
D_{t}Z = \sum_{m=1}^{3} \left\{ \dfrac{d z_m}{dt}
+ \sum_{i,j=1}^{3}\Gamma_{ij}^{m}z_{j} \dfrac{d x_i}{dt}\right\} e_m,
\end{equation*}
 where the $x_i$'s are the coordinates of $\gamma$ in a local chart
and $\Gamma_{ij}^{m}$ are the Christoffel symbols introduced before.

In particular, if $Z = \dot{\gamma}$, we have
\begin{equation}\label{eq:DCov}
D_{t}\dot{\gamma} = \sum_{m=1}^{3} \left\{ \ddot{\gamma}_{m}
+ \sum_{i,j=1}^{3} \Gamma_{ij}^{m} \dot{\gamma}_{i}\dot{\gamma}_{j}\right\}e_m .
\end{equation}

\section{Riemannian approximation of curvature of curves}\label{C1}
Let us define the objects we are going to study in this section.

\begin{definition}
Let $\gamma :[a,b] \rightarrow (\reals^{3},g_{L})$ be
a Euclidean $C^{1}$-smooth curve.
We say that $\gamma$ is regular if
$$\dot{\gamma}(t) \neq 0, \quad \textrm{for every $t \in [a,b]$}.$$
Moreover, we say that $\gamma(t)$ is a horizontal point of $\gamma$ if
$$\omega(\dot{\gamma}(t))= \dot{\gamma}_{3}(t) - \dfrac{1}{2}
\left( \gamma_{1}(t)\dot{\gamma}_{2}(t)-
\gamma_{2}(t)\dot{\gamma}_{1}(t)\right) = 0.$$
\end{definition}

\begin{definition}
Let $\gamma : [a,b]\subset \reals \rightarrow (\mathbb{R}^{3}, g_L)$
 be a Euclidean $C^{2}$-smooth regular curve in
the Riemannian manifold $(\mathbb{R}^{3}, g_L)$.
The curvature $k_{\gamma}^{L}$
of $\gamma$ at $\gamma(t)$ is defined as
\begin{equation}\label{eq:LCurv}
k_{\gamma}^{L} := \sqrt{\dfrac{\|D_t \cu\|_{L}^{2}}{\|\cu\|_{L}^{4}}
- \dfrac{\scal{D_t \cu}{{\cu}}_{L}^{2}}{\|\cu\|_{L}^{6}}} .
\end{equation}
\end{definition}

We stress that the above definition is well posed, indeed by Cauchy-Schwarz,
\begin{equation*}
 \dfrac{\|D_t \cu\|_{L}^{2}}{\|\cu\|_{L}^{4}} - \dfrac{\scal{D_t \cu}{{\cu}}_{L}^{2}}{\|\cu\|_{L}^{6}}
 \geq \dfrac{\|\cu\|_{L}^{2}\|D_t \cu\|_{L}^{2} - \|\cu\|_{L}^{2}\|D_t \cu\|_{L}^{2}}{\|\cu\|_{L}^{6}} = 0.
\end{equation*}

\begin{remark}
We recall that in Riemannian geometry
the standard definition of curvature for a curve
$\gamma$ parametrized by arc length is
$k_{\gamma}^{L} := \|D_{t}\cu\|_L$. The one we gave before is
just more practical to perform computations for curves with an
arbitrary parametrization.
\end{remark}

Let us briefly recall the definition of $g_{L}$-geodesics, cf.\ \cite[Chapter 2]{CDPT}. For
a Euclidean $C^{2}$-smooth regular curve $\gamma:[a,b] \to (\reals^{3},g_{L})$, define
its \textit{penalized energy functional} $E_{L}$ to be
$$
E_{L}(\gamma):= \int_{a}^{b}\left( |\dot{\gamma}_{1}(t)|^{2}
+ |\dot{\gamma}_{2}(t)|^{2} + L \left| \omega(\dot{\gamma}(t))|\right|^{2}\right) \, dt.
$$
Using a standard variational argument, we can derive
the system of Euler-Lagrange equations for the functional $E_{L}$:
we will call $g_L$-geodesics the critical points, which are actually curves,
of the functional $E_{L}$.
In other words, we will say that $\gamma$ is a $g_L$-geodesic, if for
every $t \in [a,b]$ it holds that
\begin{equation}\label{eq:gLEuler}
\left \{
\begin{array}{l}
 \ddot{\gamma}_{1}(t)=- L \dot{\gamma}_{2}(t) \omega(\dot{\gamma}(t)) , \\
\ddot{\gamma}_{2}(t)= L \dot{\gamma}_{1}(t) \omega(\dot{\gamma}(t)),\\
(\omega(\dot{\gamma}(t))|_{\gamma(t)})'=0.
\end{array}\right.\end{equation}

We are now ready to present the first result concerning the curvature $k_{\gamma}^{L}$.

\begin{lemma}\label{kL}
Let $\gamma : [a,b] \rightarrow (\mathbb{R}^{3}, g_L)$
be a Euclidean $C^{2}$-smooth regular curve in
the Riemannian manifold $(\mathbb{R}^{3}, g_L)$. Then
\begin{equation}\label{eq:kL}
 k_{\gamma}^{L} = \sqrt{\dfrac{(\ddot{\gamma}_{1}+ L \dot{\gamma}_2 \omega(\dot{\gamma}))^{2}
+ (\ddot{\gamma}_{1}- L \dot{\gamma}_1 \omega(\dot{\gamma}))^{2}
+ L \omega({\ddot{\gamma}})^{2}}{(\dot{\gamma}_{1}^{2}+\dot{\gamma}_{2}^{2}
+L\omega(\dot{\gamma})^{2})^{2}}
- \dfrac{(\ddot{\gamma}_1 \dot{\gamma}_1 + \ddot{\gamma}_2 \dot{\gamma}_2
+ L \omega(\dot{\gamma}) \omega(\ddot{\gamma}))^{2}}{(\dot{\gamma}_{1}^{2}
+\dot{\gamma}_{2}^{2}+L\omega(\dot{\gamma})^{2})^{3}}}.
\end{equation}
In particular, if $\gamma(t)$ is a horizontal point of $\gamma$,
\begin{equation}\label{eq:kL2}
k_{\gamma}^{L}
= \sqrt{\dfrac{\ddot{\gamma}_{1}^{2}
+ \ddot{\gamma}_{1}^{2}}{(\dot{\gamma}_{1}^{2}+\dot{\gamma}_{2}^{2})^{2}}
- \dfrac{(\ddot{\gamma}_1 \dot{\gamma}_1 + \ddot{\gamma}_2 \dot{\gamma}_2 )^{2}}{(\dot{\gamma}_{1}^{2}+\dot{\gamma}_{2}^{2})^{3}}}
= \frac{|\ddot{\gamma}_2 \dot{\gamma}_1 - \ddot{\gamma}_1 \dot{\gamma}_2|}{(\dot{\gamma}_{1}^{2}+\dot{\gamma}_{2}^{2})^{3/2}}.
\end{equation}
\end{lemma}

\begin{proof}
We first compute the covariant derivative of $\dot{\gamma}$ as in \eqref{eq:DCov},
using \eqref{eq:G1}, \eqref{eq:G2} and \eqref{eq:G3}. In components, with respect
to the standard basis of $\mathbb{R}^{3}$, we have
\begin{equation}
\begin{aligned}
(D_t \dot{\gamma})_{1} &= \ddot{\gamma}_{1} + \dfrac{\gamma_2 L}{2} \dot{\gamma}_{1} \dot{\gamma}_{2}- \dfrac{\gamma_1 L}{2} \dot{\gamma}_{2}^{2} +L \dot{\gamma}_{3} \dot{\gamma}_{2}
= \ddot{\gamma}_{1}+ L \dot{\gamma}_2 \omega(\dot{\gamma}),\\
(D_t \dot{\gamma})_{2} &= \ddot{\gamma}_{2} - \dfrac{\gamma_2 L}{2} \dot{\gamma}_{1}^{2} + \dfrac{\gamma_1 L}{2}  \dot{\gamma}_{1}\dot{\gamma}_{2}-L \dot{\gamma}_{3} \dot{\gamma}_{1}
= \ddot{\gamma}_{2}- L \dot{\gamma}_1 \omega(\dot{\gamma}),\\
(D_t \dot{\gamma})_{3} &= \ddot{\gamma}_{3} - \dfrac{\gamma_1 \gamma_2 L}{4} \dot{\gamma}_{1}^{2} + \dfrac{(\gamma_{1}^{2}-\gamma_{2}^{2})L}{4}\dot{\gamma}_{1}\dot{\gamma}_{2} - \dfrac{\gamma_1 L}{2}  \dot{\gamma}_{1}\dot{\gamma}_{3}
+\dfrac{\gamma_1 \gamma_2 L}{4} \dot{\gamma}_{2}^{2} - \dfrac{\gamma_2 L}{2}  \dot{\gamma}_{2}\dot{\gamma}_{3}.
\end{aligned}\end{equation}
Now we express $D_t \dot{\gamma}$ in the basis $\{X_1,X_2,\XtL \}$:
\begin{equation}\label{eq:DC}
D_{t}\dot{\gamma} =
\bigl( \ddot{\gamma}_{1}+ L \dot{\gamma}_2 \omega(\dot{\gamma}) \bigr) \, X_1
+
\bigl( \ddot{\gamma}_{2}- L \dot{\gamma}_1 \omega(\dot{\gamma}) \bigr) \, X_2
+
\left( \sqrt{L} \omega({\ddot{\gamma}})\right) \, \XtL,
\end{equation}
where
$$
\omega(\ddot{\gamma}(t)) = \ddot{\gamma}_{3}(t) - \dfrac{1}{2}\left( \gamma_{1}(t)\ddot{\gamma}_{2}(t) - \gamma_{2}(t)\ddot{\gamma}_{1}(t)\right),
$$
coincides with the expression $(\omega(\dot{\gamma}))'$ in \eqref{eq:gLEuler}.
Recalling that
\begin{equation}\label{dot-gamma-in-coordinates}
\dot\gamma = \dot\gamma_1 \, X_1 + \dot\gamma_2 \, X_2 + (\sqrt{L}\omega(\dot\gamma)) \, \XtL,
\end{equation}
we compute
$$
\scal{D_t \dot{\gamma}}{\dot{\gamma}}_{L} = \ddot{\gamma}_1 \dot{\gamma}_1 + \ddot{\gamma}_2 \dot{\gamma}_2 + L \omega(\dot{\gamma}) \omega(\ddot{\gamma}).
$$
Therefore
\begin{equation}\label{eq:Curvature}
  k_{\gamma}^{L} = \sqrt{\dfrac{(\ddot{\gamma}_{1}+ L \dot{\gamma}_2 \omega(\dot{\gamma}))^{2}
+ (\ddot{\gamma}_{2}- L \dot{\gamma}_1 \omega(\dot{\gamma}))^{2} + L \omega({\ddot{\gamma}})^{2}}{(\dot{\gamma}_{1}^{2}+\dot{\gamma}_{2}^{2}+L\omega(\dot{\gamma})^{2})^{2}}
- \dfrac{(\ddot{\gamma}_1 \dot{\gamma}_1 + \ddot{\gamma}_2 \dot{\gamma}_2 + L \omega(\dot{\gamma}) \omega(\ddot{\gamma}))^{2}}{(\dot{\gamma}_{1}^{2}+\dot{\gamma}_{2}^{2}+L\omega(\dot{\gamma})^{2})^{3}}}.
\end{equation}
On the other hand, if $\gamma(t)$ is a horizontal point for $\gamma$, then $\omega(\dot{\gamma}(t))=0$ and
$$
D_{t}\dot{\gamma} = \ddot\gamma_1 \, X_1 + \ddot\gamma_2 \, X_2.
$$
It follows that
$$
k_{\gamma}^{L} = \sqrt{\dfrac{\ddot{\gamma}_{1}^{2}
+ \ddot{\gamma}_{1}^{2}}{(\dot{\gamma}_{1}^{2}+\dot{\gamma}_{2}^{2})^{2}}
- \dfrac{(\ddot{\gamma}_1 \dot{\gamma}_1 + \ddot{\gamma}_2 \dot{\gamma}_2 )^{2}}{(\dot{\gamma}_{1}^{2}+\dot{\gamma}_{2}^{2})^{3}}},
$$
as desired.
\end{proof}

\begin{remark}\label{Remark}
The pointwise notion of curvature provided by (\ref{eq:kL}) is continuous along the
curve $\gamma$. Moreover, we want to stress that this notion is independent of
the parametrization chosen.
\end{remark}

We now use the previous results to study curvatures of curves in the Heisenberg group $\heis$.

\begin{definition}\label{def:sRc}
Let $\gamma : [a, b] \rightarrow \heis$ be a Euclidean $C^{2}$-smooth regular curve.
We define the sub-Riemannian curvature
$k_{\gamma}^{0}$ of $\gamma$ at $\gamma(t)$ to be
$$
k_{\gamma}^{0} := \lim_{L \rightarrow +\infty} k_{\gamma}^{L},
$$
if the limit exists.
\end{definition}

It is clear that, a priori, the above limit may not exist.
In order to deal with the asymptotics as $L \to +\infty$ of
the quantities involved, let us introduce the
following notation: for continuous functions $f,g: (0, +\infty) \to \reals$,
\begin{equation}\label{eq:SIM}
f(L) \sim g(L), \quad \textrm{as $L \to +\infty$} \quad \stackrel{def}{\Longleftrightarrow}  \quad
\lim_{L \to +\infty} \dfrac{f(L)}{g(L)}=1.
\end{equation}

\begin{lemma}\label{k0}
Let $\gamma : [a, b] \rightarrow \heis$ be a Euclidean $C^{2}$-smooth regular curve. Then
\begin{equation}\label{eq:k0}
k_{\gamma}^{0} =
\left \{
\begin{array}{rl}
\dfrac{|\cuo \ddot{\gamma}_2 - \cut \ddot{\gamma}_1|}{\sqrt{(\cuo^2 + \cut^2)^{3}}}, &  \mbox{if $\gamma(t)$ is a horizontal point of $\gamma$,}  \\
\dfrac{\sqrt{\cuo^{2}+ \cut^{2}}}{|\omega(\dot{\gamma})|} , & \mbox{if $\gamma(t)$ is not a horizontal point of $\gamma$.}
\end{array}\right.
\end{equation}
\end{lemma}

\begin{proof}
The first result follows from the fact that the
expression (\ref{eq:kL2}) for the curvature at horizontal points
in $(\mathbb{R}^{3}, g_L)$, is independent of $L$. For the other case, we need to study the asymptotics
in $L$. Using the notation introduced in \eqref{eq:SIM}, we have
\begin{align*}
\|D_t \cu\|_{L} \sim L  \sqrt{(\cuo^{2}+ \cut^{2})} |\cf|, \quad \mbox{as $L \to +\infty$},\\
\|\cu \|_{L} \sim \sqrt{L} \,|\cf|, \quad \mbox{as $L \to +\infty$},\\
\scal{D_t \cu}{\cu}_{L} \sim L \,\cf \, \omega(\ddot{\gamma}), \quad \mbox{as $L \to +\infty$}.
\end{align*}
Therefore
\begin{align*}
 \dfrac{\|D_t \cu \|_{L}^{2}}{\|\cu\|_{L}^{4}} &\to \dfrac{\cuo^{2}+ \cut^{2}}{\cf^{2}}, \quad \mbox{as $L \to +\infty$},\\
\dfrac{\scal{D_t \cu}{\cu}_{L}^{2}}{\|\cu\|_{L}^{6}} &\sim \dfrac{L^{2} \cf^{2} \omega(\ddot{\gamma})^{2}}{L^{3} \cf^{6}} \to 0, \quad \mbox{as $L \to +\infty$}.
\end{align*}
Altogether,
$$k_{\gamma}^{0} = \lim_{L \to +\infty} k_{\gamma}^{L}
= \dfrac{\sqrt{\cuo^{2}+ \cut^{2}}}{|\omega(\dot{\gamma})|},$$
as desired.
\end{proof}

\begin{remark}
We stress that the $g_L$-length of a curve can, in general, blow
up, once we take the limit as $L \to +\infty$.
The remarkable fact provided by Lemma \ref{k0} is that this never occurs
for the curvature $k_{\gamma}^{L}$. We notice
that, as in Remark \ref{Remark}, the sub-Riemannian
curvature $k_{\gamma}^{0}$ is independent of the parametrization. In the horizontal case, the quantity $k_\gamma^0$ is the absolute value of the {\it signed horizontal curvature} of a horizontal curve which arises in the study of horizontal mean curvature. See Remark \ref{horiz-mean-curve-remark} for more details.
\end{remark}

There is another fact to notice.
If we consider a Euclidean $C^{2}$-smooth regular
curve $\gamma :[a,b] \to \heis$, which is
partially horizontal and partially not, the quantity $k_{\gamma}^{0}(t)$ in
\eqref{eq:k0}, need not be a continuous function (in contrast to Remark \ref{Remark}).
Moreover, in view of the independence of $k_{\gamma}^{0}$ of the parametrization,
when we approach a horizontal point of $\gamma$ from non-horizontal points of $\gamma$, we always find a singularity.
Let us clarify the last sentences with an example.

\begin{example}
 Consider the planar curve
 $$\gamma(\theta):= \left( \cos(\theta)+1, \sin(\theta),0\right), \quad \theta \in [0,2\pi).$$
 The curve $\gamma$ is horizontal only for $\theta = \pi$, indeed
 $\omega(\cu)= -\tfrac{1+\cos(\theta)}{2}$, which vanishes only for $\theta = \pi$. The
 horizontal curvature $k_{\gamma}^{0}$ is a pointwise notion, therefore we have
 \begin{equation*}
  k_{\gamma}^{0}(\theta)|_{\theta=\pi} =
  \sqrt{\dfrac{\ddot{\gamma}_{1}^{2}
+ \ddot{\gamma}_{1}^{2}}{(\dot{\gamma}_{1}^{2}+\dot{\gamma}_{2}^{2})^{2}}
- \dfrac{(\ddot{\gamma}_1 \dot{\gamma}_1 + \ddot{\gamma}_2 \dot{\gamma}_2 )^{2}}{(\dot{\gamma}_{1}^{2}+\dot{\gamma}_{2}^{2})^{3}}} \bigg|_{\theta = \pi} = 1,
 \end{equation*}
 and for $\theta \in [0,2\pi)\setminus \{\pi\}$,
\begin{equation*}
k_{\gamma}^{0}(\theta)= \dfrac{\sqrt{\cuo^{2}(\theta)+ \cut^{2}(\theta)}}{|\omega(\dot{\gamma}(\theta))|} = \dfrac{2}{|1+\cos(\theta)|} \to +\infty, \quad \textrm{as $\theta \to \pi$}.
\end{equation*}
\end{example}

In the classical theory of differential geometry of smooth space curves in $\reals^{3}$,
there is a famous \textit{rigidity theorem} (see for instance \cite{dC76}) stating that
every curve is characterized by its curvature and its torsion, up to rigid motions.
Similar questions are addressed in \cite{CL13} and \cite{CL15} but with a different approach,
viewing the first Heisenberg group $\heis$ as flat 3-dimensional manifold carrying a pseudo-hermitian
structure.

We now have a notion of \textit{horizontal curvature} for Euclidean $C^{2}$-smooth regular
curves in $(\heis ,d_{cc})$. A first result in the direction of the classical rigidity
theorem is the following.

\begin{proposition}\label{curvature-isometry-proposition}
Let $\gamma :[a,b]\to (\heis,d_{cc})$ be a Euclidean $C^{2}$-smooth regular
curve. The sub-Riemannian curvature $k_{\gamma}^{0}$ of $\gamma$ is invariant
under left translation and rotations around the $x_3$-axis of $\gamma$.
\end{proposition}

\begin{proof}
Fix a point $g = (g_1, g_2,g_3) \in \heis$. Define the curve $\tilde{\gamma}$ as
the left translation by $g$ of $\gamma$,
$$
\tilde{\gamma}(t):= L_{g}(\gamma(t)) = \left(\gamma_1(t)+ g_1, \gamma_2(t) + g_2, \gamma_3(t)+g_3 - \dfrac{1}{2}(\gamma_1(t) g_2 - \gamma_2(t) g_1)\right), \quad t\in [a,b].
$$
It is clear that
$$
\dot{\tilde{\gamma}}_{i}=\dot{\gamma}_{i} \quad \mbox{and} \quad \ddot{\tilde{\gamma}}_{i}=\ddot{\gamma}_{i}, \quad i=1,2,
$$
and therefore $k_{\gamma}^{0} = k_{\tilde{\gamma}}^{0}$.

For the second assertion, fix an angle $\theta \in [0,2\pi)$ and define the curve $\bar{\gamma}$ as the
rotation by $\theta$ of $\gamma$ around the $x_3$-axis,
$$
\bar{\gamma}(t):= \left(\gamma_1(t)\cos(\theta)+ \gamma_2(t)\sin(\theta), -\gamma_{1}\sin(\theta)+\gamma_{2}\cos(\theta), \gamma_3 \right), \quad t\in [a,b].
$$
An easy computation shows that in this case we have
$$
\dot{\bar{\gamma}}_{1}^{2}+\dot{\bar{\gamma}}_{2}^{2}=\dot{\gamma}_{1}^{2} +\dot{\gamma}_{2}^{2}, \quad
\ddot{\bar{\gamma}}_{1}^{2}+\ddot{\bar{\gamma}}_{2}^{2}=\ddot{\gamma}_{1}^{2} +\ddot{\gamma}_{2}^{2}, \quad
\dot{\bar{\gamma}}_{1}\ddot{\bar{\gamma}}_{1}+\dot{\bar{\gamma}}_{2}\ddot{\bar{\gamma}}_{2} = \dot{\gamma}_{1}\ddot{\gamma}_{1}+\dot{\gamma}_{2}\ddot{\gamma}_{2},
$$
and therefore $k_{\gamma}^{0} = k_{\bar{\gamma}}^{0}$.
\end{proof}

\begin{remark}\label{curvature-dilation-remark}
The behavior of the curvature $k_\gamma^0$ under dilations is as follows: if $\gamma$ is a $C^2$ smooth  regular curve and $r>0$, then
$$
k_{\delta_r\gamma}^0 = \frac1r k_\gamma^0.
$$
Here $\delta_r\gamma$ denotes the curve $(r\gamma_1,r\gamma_2,r^2\gamma_3)$ when $\gamma=(\gamma_1,\gamma_2,\gamma_3)$.
\end{remark}

\section{Riemannian approximation of geodesic curvature of curves on surfaces}\label{C2}
In order to prove a Heisenberg version of the Gauss--Bonnet Theorem,
we need the concept of \textit{sub-Riemannian signed geodesic curvature}.
This section will be devoted to the study of curvature of curves living on surfaces.

Let us fix once for all the assumptions we will make on the surface $\Sigma$ in this and the coming section.
We will say that a surface $\Sigma \subset (\reals^{3},g_{L})$, or $\Sigma \subset \heis$, is \textbf{regular} if
\begin{equation}\label{eq:Sigma}
\Sigma \quad \textrm{is a Euclidean $C^{2}$-smooth compact and oriented surface}.
\end{equation}
In particular we will assume that there
exists a Euclidean $C^{2}$-smooth function $u:\reals^{3}\to \reals$ such that
$$
\Sigma = \{(x_1,x_2,x_3)\in \reals^{3}: u(x_1,x_2,x_3)=0\}
$$
and $\nabla_{\reals^{3}}u \neq 0$.
As in \cite[Section 4.2]{CDPT}, our study will be local and away from characteristic points of $\Sigma$. For completeness, we
recall that a point $x \in \Sigma$ is called \textit{characteristic} if
\begin{equation}\label{eq:char}
\hgrad u(x) = (0,0).
\end{equation}
The presence or absence of characteristic points will be stated explicitly.

To fix notation (following the one adopted in \cite[Chapter 4]{CDPT}), let us define first
$$
p:= X_1 u, \quad q:= X_2 u \quad  \textrm{and} \quad r:= X_{3}^{L} u.
$$
We then define
\begin{equation}\label{eq:pq}
\begin{aligned}
&l:=\|\hgrad u\|_{\he}, \quad \bar{p}:= \dfrac{p}{l} \quad \textrm{and} \quad \bar{q}:= \dfrac{q}{l}, \\
&l_{L}:= \sqrt{(X_{1}u)^{2}+(X_{2}u)^{2}+(\XtL u)^{2}}, \quad \bar{r}_L:= \dfrac{r}{l_{L}},\\
&\bar{p}_{L}:= \dfrac{p}{l_L} \quad \mbox{and} \quad \bar{q}_{L}:= \dfrac{q}{l_L}.
\end{aligned}
\end{equation}
In particular, $\bar{p}^{2}+\bar{q}^{2}=1$. It is clear from \eqref{eq:char} that
these functions are well defined at every non-characteristic point of $\Sigma$.

\begin{definition}
Let $\Sigma \subset (\mathbb{R}^{3}, g_{L})$ be a regular surface and let $u: \mathbb{R}^{3} \to \mathbb{R}$
be as before. The Riemannian unit normal $\nu_{L}$ to $\Sigma$ is
$$\nu_{L} := \dfrac{\grad_{L}u}{\|\grad_{L}u\|_{L}} = \bar{p}_{L}X_1 + \bar{q}_{L}X_2 + \bar{r}_{L}X_{3}^{L},$$
where $\grad_{L}u$ is the Riemannian gradient of $u$.
\end{definition}

\begin{definition}\label{E1E2}
Let $\Sigma \subset (\mathbb{R}^{3}, g_{L})$ be a regular surface
and let $u: \mathbb{R}^{3} \to \mathbb{R}$ be as before.
For every point $g \in \Sigma$, we introduce the orthonormal basis
$\{E_{1}(g), E_{2}(g)\}$ for $T_g\Sigma$, where
$$
E_{1}(g) :=   \bar{q}(g) \, X_{1}(g) - \bar{p}(g) \, X_{2}(g)
$$
and
$$
E_{2}(g) := \bar{r}_{L}(g) \, \bar{p}(g) \, X_{1}(g) +
\bar{r}_{L}(g) \, \bar{q}(g)\, X_{2}(g) - \dfrac{l}{l_L}(g) \, \XtL(g) .$$
On $T_g \Sigma$ we define a linear transformation $J_L : T_g \Sigma \to T_g \Sigma$ such that
\begin{equation}\label{eq:DefJ}
J_L (E_{1}(g)) := E_{2}(g), \quad \mbox{and} \quad J_L(E_{2}(g)):= -E_{1}(g).
\end{equation}
In the following we will omit the dependence on the point $g \in \Sigma$ when it is clear.
\end{definition}

Let us spend a few words on the choice of the basis of the tangent plane $T_{g}\Sigma$.
The vector $E_1$ is a horizontal vector given by $J(\hgrad u / \|\hgrad u\|_H)$, where $J$ is the linear operator acting on horizontal vector fields by $J(aX_1+bX_2) = bX_1-aX_2$.
This vector is called by some authors
\textit{characteristic direction} and plays an important role in the study of minimal surfaces
in $\heis$ and in the study
of the properties of the characteristic set $\mathrm{char}(\Sigma)$,
see for instance \cite{CHMY05}, \cite{CHMY12} and \cite{CH04}.

In order to perform computations on the surface $\Sigma$, we need
a Riemannian metric and a connection.
Classically, see for instance \cite{L} or \cite{dC},
there is a standard way to define a Riemannian metric
$g_{L, \Sigma}$ on $\Sigma$ so that $\Sigma$ is
isometrically immersed in $(\mathbb{R}^{3}, g_L)$.
Once we have a Riemannian metric on the surface $\Sigma$,
we can define the unique Levi-Civita connection $\CoSi$ on $\Sigma$
related to the Riemannian metric $g_{L, \Sigma}$.
This procedure is equivalent to the following definition.

\begin{definition}
Let $\Sigma \subset (\mathbb{R}^{3}, g_L)$ be a Euclidean $C^{2}$-smooth surface.
For every $U,V \in T_{g}\Sigma$
we define $\CoSi_{U}V$ to be the tangential component of $\nabla$, namely
$$\CoSi_{U}V = \Pi \left(\nabla_{U}V\right),$$
where $\Pi: \mathbb{R}^{3} \rightarrow T\Sigma$.
\end{definition}

Note that to compute $\nabla_{U}V$ we are considering an extension of both $U$ and $V$ to $\mathbb{R}^{3}$.

As a notational remark, given a
Euclidean $C^{2}$-smooth and regular curve $\gamma \subset \Sigma$,
we will denote the covariant derivative of $\cu$
with respect to the $g_{L,\Sigma}$-metric,
as $D_{t}^{\Sigma} \cu$.

We are now interested in detecting the curvature of a
Euclidean $C^{2}$-smooth regular curve $\gamma$
on a given regular surface $\Sigma \subset \heis$.

We start with a technical lemma that will simplify our treatment in the ensuing discussion.

\begin{lemma}\label{limits}
For $\bar{p}$, $\bar{q}$, $l_{L}$ and $\bar{r}_L$ as before, we have
\begin{align}
l_{L} \to \|\hgrad u\|, & \quad \mbox{as $L \to +\infty$},\\
\bar{r}_{L} \to 0,& \quad \mbox{as $L \to +\infty$}\label{eq:2lim},\\
\dfrac{\bar{r}_{L}}{l_{L}} \to 0,& \quad \mbox{as $L \to +\infty$}, \\
\dfrac{\sqrt{L} \, \bar{r}_{L}}{l_{L}} \to \dfrac{X_3 u}{\|\hgrad u\|^{2}},& \quad \mbox{as $L \to +\infty$}, \\
L \,\bar{r}_{L}^{2} \to  \dfrac{(X_3 u)^{2}}{\|\hgrad u\|^{2}},& \quad \mbox{as $L \to +\infty$},\\
\bar{r}_{L}L \sim  \dfrac{\sqrt{L}(X_{3}u)}{\|\hgrad u\|},& \quad \mbox{as $L \to +\infty$}\label{eq:6lim}
\end{align}
\end{lemma}

\begin{proof}
All the limits and asymptotics follow directly from the definitions in (\ref{eq:pq}).
\end{proof}

\begin{lemma}\label{NonHor}
Let $\gamma: [a,b] \rightarrow \Sigma$ be a Euclidean $C^{2}$-smooth, regular and non-horizontal curve.
The covariant derivative $D^{\Sigma}_t \cu$ of $\cu$
with respect to the $g_{L, \Sigma}$-metric is given by
\begin{equation*}
 D^{\Sigma}_t \cu = (D^{\Sigma}_t \cu)_1 E_1 + (D^{\Sigma}_t \cu)_2 E_2,
\end{equation*}
 where
\begin{equation}\label{eq:DEi}
\begin{aligned}
(D^{\Sigma}_t \cu)_1 &= \bar{q} \left(\ddot{\gamma}_1 + L \omega(\cu) \cut \right)- \bar{p}\left( \ddot{\gamma}_2 - L\omega(\cu) \cuo \right),\\
(D^{\Sigma}_t \cu)_2 &= \bar{r}_L \bar{p} \left(\ddot{\gamma}_1 + L \omega(\cu) \cut \right)+ \bar{r}_L \bar{q}\left( \ddot{\gamma}_2 - L\omega(\cu) \cuo \right) - \dfrac{l}{l_L}\sqrt{L} \omega(\ddot{\gamma}).
\end{aligned}
\end{equation}
Moreover, if $\gamma: [a,b] \rightarrow \Sigma$ is Euclidean $C^{2}$-smooth, regular and horizontal curve,
we have
\begin{equation}\label{eq:DEii}
 D^{\Sigma}_t \cu = \left(\bar{q} \ddot{\gamma}_1- \bar{p} \ddot{\gamma}_2 \right) E_1 + \bar{r}_L \, \left(\bar{p} \ddot{\gamma}_1 + \bar{q}\ddot{\gamma}_2\right) E_2,
\end{equation}
\end{lemma}

\begin{proof}
The covariant derivative $D_t \cu$ in the $\{X_1,X_2,X_{3}^{L}\}$ basis is as in (\ref{eq:DC}). Projecting $D_t\cu$ via $\Pi$ inon the tangent plane $T_g \Sigma$, we find
\begin{equation*}
\begin{aligned}
 D^{\Sigma}_t \cu :&= \scal{D_t \cu}{E_1}_L E_1 + \scal{D_t \cu}{E_2}_L E_2 \\
]&= \left[ \bar{q} \left(\ddot{\gamma}_1 + L \omega(\cu) \cut \right)- \bar{p}\left( \ddot{\gamma}_2 - L\omega(\cu) \cuo \right)\right) E_1 \\
& \quad + \left[ \bar{r}_L \bar{p} \left(\ddot{\gamma}_1 + L \omega(\cu) \cut \right)+ \bar{r}_L \bar{q}\left( \ddot{\gamma}_2 - L\omega(\cu) \cuo \right) - \dfrac{l}{l_L}\sqrt{L} \omega(\ddot{\gamma})\right] E_2.
\end{aligned}\end{equation*}
The situation is much simpler if $\gamma$ is horizontal, indeed
it suffices to set $\omega(\cu) = \omega(\ddot{\gamma})=0$ in the
previous expression.
\end{proof}

\begin{definition}
Let $\Sigma \subset (\mathbb{R}^{3}, g_L)$ be a regular surface.
Let $\gamma : [a, b] \rightarrow \Sigma$ be
a Euclidean $C^{2}$-smooth and regular curve, and
let $D_{t}^{\Sigma}$ be the covariant derivative of $\cu$
with respect to the Riemannian metric $g_{L,\Sigma}$.
The geodesic curvature $k_{\gamma, \Sigma}^{L}$ of
$\gamma$ at the point $\gamma(t)$ is defined to be
$$
k_{\gamma, \Sigma}^{L} := \sqrt{ \dfrac{\| D^{\Sigma}_t \cu \|^{2}_{L,\Sigma}}{\|\cu\|_{L,\Sigma}^{4}} - \dfrac{\scal{D^{\Sigma}_t \cu}{{\cu}}^{2}_{L,\Sigma}}{\|\cu\|_{L,\Sigma}^{6}}}.
$$
\end{definition}

\begin{definition}
Let $\Sigma \subset \heis$ be a regular surface and let
$\gamma : [a, b] \rightarrow \Sigma$ be a Euclidean $C^{2}$-smooth and regular curve.
The \textit{sub-Riemannian geodesic curvature} $k_{\gamma, \Sigma}^{0}$ of
$\gamma$ at the point $\gamma(t)$ is defined as
$$
k_{\gamma, \Sigma}^{0}:= \lim_{L \rightarrow +\infty} k_{\gamma, \Sigma}^{L},
$$
if the limit exists.
\end{definition}

As in Section \ref{C1}, our first task is to show that the above limit does always exist.

\begin{lemma}\label{GeoCur}
Let $\Sigma \subset \heis$ be a regular surface and let $\gamma : [a, b] \rightarrow \Sigma$ be a
Euclidean $C^{2}$-smooth and regular curve. Then
$$
k_{\gamma, \Sigma}^{0} = \dfrac{|\bar{p} \cuo + \bar{q} \cut |}{|\omega(\cu)|}
$$
if $\gamma(t)$ is a non-horizontal point, while
$$
k_{\gamma, \Sigma}^{0} = 0
$$
if $\gamma(t)$ is a horizontal point.
\end{lemma}

\begin{proof}
First, let $\gamma(t)$ be a non-horizontal point of the curve $\gamma$.
From Lemma \ref{NonHor}, we know that
\begin{equation*}
\begin{aligned}
 D^{\Sigma}_t \cu &=\left( \bar{q} \left(\ddot{\gamma}_1 + L \omega(\cu) \cut \right)- \bar{p}\left( \ddot{\gamma}_2 - L\omega(\cu) \cuo \right)\right) E_1 \\
& \quad + \left( \bar{r}_L \bar{p} \left(\ddot{\gamma}_1 + L \omega(\cu) \cut \right)+ \bar{r}_L \bar{q}\left( \ddot{\gamma}_2 - L\omega(\cu) \cuo \right) - \dfrac{l}{l_L}\sqrt{L} \omega(\ddot{\gamma})\right) E_2.
\end{aligned}\end{equation*}
For $\cu$, recalling \eqref{dot-gamma-in-coordinates} we have
\begin{equation*}
\cu =\left( \bar{q} \cuo - \bar{p}\cut \right) E_1 + \left( \bar{r}_L \bar{p} \cuo + \bar{r}_L \bar{q} \cut - \dfrac{l}{l_L}\sqrt{L} \omega(\dot{\gamma})\right) E_2.
\end{equation*}
Recalling \eqref{eq:6lim}, we have
\begin{equation*}
 \begin{aligned}
\|D_{t}^{\Sigma}\cu\|_{L,\Sigma}^{2} &=  \left(\bar{q} \left(\ddot{\gamma}_1 + L \omega(\cu) \cut \right)- \bar{p}\left( \ddot{\gamma}_2 - L\omega(\cu) \cuo \right) \right)^{2} \\
& \quad + \left( \bar{r}_L \bar{p} \left(\ddot{\gamma}_1 + L \omega(\cu) \cut \right)+ \bar{r}_L \bar{q}\left( \ddot{\gamma}_2 - L\omega(\cu) \cuo \right) - \dfrac{l}{l_L}\sqrt{L} \omega(\ddot{\gamma})\right)^{2}\\
& \sim L^{2} \omega(\cu)^{2} (\bar{p}\cuo + \bar{q}\cut)^{2}, \quad \mbox{as $L \to +\infty$}.
 \end{aligned}
\end{equation*}
In a similar way, we have that
\begin{equation*}
 \begin{aligned}
\|\cu\|_{L,\Sigma} &=\sqrt{\left( \bar{q} \cuo - \bar{p}\cut \right)^{2} + \left( \bar{r}_L \bar{p} \cuo + \bar{r}_L \bar{q} \cut - \dfrac{l}{l_L}\sqrt{L} \omega(\dot{\gamma})\right)^{2}}\\
& \quad \sim \sqrt{L} |\omega(\cu)|, \quad \mbox{as $L \to +\infty$}.
 \end{aligned}
\end{equation*}
A bit more involved computation shows that
$$
\scal{D_{t}^{\Sigma}\cu}{\cu}_{L,\Sigma} \sim L\, M, \quad \mbox{as $L \to +\infty$},
$$
where $M$ does not depend on $L$.
Therefore, at a non-horizontal point $\gamma(t)$ , we have
$$
k_{\gamma,\Sigma}^{0} = \dfrac{|\bar{p} \cuo + \bar{q} \cut |}{|\omega(\cu)|}.
$$
Now let $\gamma(t)$ be a horizontal point of the curve $\gamma$. In this case,
\begin{align*}
D^{\Sigma}_t \cu &=\left(\bar{q}\ddot{\gamma}_1 - \bar{p} \ddot{\gamma}_2 \right) E_1 + \bar{r}_L  \left( \bar{p} \ddot{\gamma}_1 +\bar{q} \ddot{\gamma}_2\right) E_2,\\
\cu & =  \left( \bar{q} \cuo - \bar{p}\cut \right) E_1 + \bar{r}_L \, \left( \bar{p} \cuo + \bar{q} \cut \right) E_2.
\end{align*}
Recalling (\ref{eq:2lim}), we have
\begin{align*}
\|D_{t}^{\Sigma}\cu\|_{L,\Sigma}^{2} &\to \left(\bar{q}\ddot{\gamma}_1 - \bar{p} \ddot{\gamma}_2 \right)^{2}, \quad \mbox{as $L \to +\infty$,}\\
\|\cu\| _{L,\Sigma}^{2} &\to (\bar{q} \cuo - \bar{p}\cut)^{2} , \quad \mbox{as $L \to +\infty$,}\\
\scal{D_{t}^{\Sigma}\cu}{\cu}_{L,\Sigma} & \to (\bar{q}\cuo - \bar{p}\cut) (\bar{q} \ddot{\gamma}_1 -\bar{p} \ddot{\gamma}_2), \quad \mbox{as $L \to +\infty$.}
\end{align*}
Therefore
$$k_{\gamma,\Sigma}^{0}= \sqrt{\dfrac{\left(\bar{q}\ddot{\gamma}_1 - \bar{p} \ddot{\gamma}_2 \right)^{2}}{(\bar{q} \cuo - \bar{p}\cut)^{4}} - \dfrac{(\bar{q}\cuo - \bar{p}\cut)^{2} (\bar{q} \ddot{\gamma}_1 -\bar{p} \ddot{\gamma}_2)^{2}}{(\bar{q} \cuo - \bar{p}\cut)^{6}} } = 0,$$
as desired.
\end{proof}

\begin{remark}
 In the last computation, it is possible to divide by the term $\bar{q} \cuo - \bar{p}\cut$,
because it cannot be $0$ when $\gamma(t)$ is a horizontal point. Indeed, let $g \in \Sigma$ be
a non-characteristic point, then we know that
$$\dim (H_g \cap T_{g}\Sigma) =1.$$
At a horizontal point $\gamma(t)$, we have that
$\dot{\gamma} \in H_g \cup T_{g}\Sigma$. On the other hand, $E_1 \in H_g \cap T_{g}\Sigma$ as well.
Therefore, if $\bar{q} \cuo - \bar{p}\cut = 0$, this would imply that we would
have
$$\scal{\dot{\gamma}}{E_1}_{L} = 0,$$
 and therefore that either $\dim (H_g \cap T_{g}\Sigma) =2$ or $\dot{\gamma}=0$,
which would lead to a contradiction.
\end{remark}

For planar curves it is possible to define a notion of \textit{signed curvature}. Actually, it is possible to define an analogous concept
for curves living on arbitrary two-dimensional Riemannian manifolds (see for instance \cite[Chapter 9]{L}.

\begin{definition}
Let $\Sigma \subset (\mathbb{R}^{3}, g_L)$ be a regular surface and let $\gamma : [a,b] \rightarrow \Sigma$
be a Euclidean $C^{2}$-smooth and regular curve.
Let $\CoSi$ be the Levi-Civita connection on $\Sigma$
related to the Riemannian metric $g_{L,\Sigma}$.
For every $g \in \Sigma$, let $\{E_1(g),E_2(g)\}$ be an orthonormal basis of $T_g \Sigma$.
The signed geodesic curvature $k_{\gamma, \Sigma}^{L,s}$ of
$\gamma$ at the point $\gamma(t)$ is defined as
$$k_{\gamma, \Sigma}^{L,s} := \dfrac{\scal{D_{t}^{\Sigma}\cu}{J_{L}(\cu)}_{L,\Sigma}}{\|\cu\|_{L,\Sigma}^{3}},$$
where $J_{L}$ is the linear transformation on $T_g \Sigma$ defined in (\ref{eq:DefJ}).
\end{definition}

\begin{remark}
If we take the absolute value of $k_{\gamma, \Sigma}^{L,s}$,
we recover precisely $k_{\gamma, \Sigma}^{L}$.
Indeed, by definition of $J_L$, $\{ \cu/\|\cu\|_{L}, J_L(\cu)/\|\cu\|_{L}\}$ is
another orthonormal basis of $T_g \Sigma$ oriented as $\{E_{1}(g), E_{2}(g)\}$.
Since $D_{t}^{\Sigma}\cu$ is defined as the projection of $D_t \cu$ onto $T_g \Sigma$, we have
$$
D_{t}^{\Sigma}\cu = \scal{D_{t}^{\Sigma}\cu}{\dfrac{\cu}{\|\cu\|_{L}}}_L \dfrac{\cu}{\|\cu\|_{L}} + \scal{D_{t}^{\Sigma}\cu}{\dfrac{J_L(\cu)}{\|\cu\|_{L}}}_L \dfrac{J_L(\cu)}{\|\cu\|_{L}}.
$$
In particular,
$$
\dfrac{\|D_{t}^{\Sigma}\cu\|_{L}^{2}}{\|\cu\|_{L}^{4}} = \dfrac{\scal{D_{t}^{\Sigma}\cu}{\cu}_{L,\Sigma}^{2}}{\|\cu\|_{L}^{6}} + \dfrac{\scal{D_{t}^{\Sigma}\cu}{J_L(\cu)}_{L}^{2}}{\|\cu\|_{L}^{6}}
$$
and so
\begin{equation*}
|k_{\gamma, \Sigma}^{L,s}| = \dfrac{|\scal{D_{t}^{\Sigma}\cu}{J_{L}(\cu)}_{L}|}{\|\cu\|_{L}^{3}}
= \sqrt{\dfrac{\|D_{t}^{\Sigma}\cu\|_{L}^{2}}{\|\cu\|_{L}^{4}}-\dfrac{\scal{D_{t}^{\Sigma}\cu}{\cu}_{L}^{2}}{\|\cu\|_{L}^{6}}} = k_{\gamma, \Sigma}^{L}.
\end{equation*}
\end{remark}

\begin{definition}
Let $\Sigma \subset \heis$ be a regular surface and let $\gamma : [a, b] \rightarrow \Sigma$ be a Euclidean $C^{2}$-smooth and regular curve.
The \textit{sub-Riemannian signed geodesic curvature} $k_{\gamma, \Sigma}^{0,s}(t)$ of
$\gamma$ at the point $\gamma(t)$ is defined as
$$k_{\gamma, \Sigma}^{0,s} := \lim_{L \rightarrow +\infty} k_{\gamma, \Sigma}^{L,s}.$$
\end{definition}

Again, we need to show that the above limit actually exists.

\begin{proposition}\label{SGeoCur}
 Let $\Sigma \subset \heis$ be a regular surface and let $\gamma : [a,b] \rightarrow \Sigma$ be a Euclidean $C^{2}$-smooth and regular curve.
Then
$$
k_{\gamma, \Sigma}^{0,s} = \dfrac{\bar{p} \cuo + \bar{q} \cut}{|\omega(\cu)|}
$$
if $\gamma(t)$ is a non-horizontal point, and
$$
k_{\gamma, \Sigma}^{0,s} = 0
$$
if $\gamma(t)$ is a horizontal point.
\end{proposition}

\begin{proof}
 We already know that
\begin{equation*}
\begin{aligned}
 D^{\Sigma}_t \cu &= \left( \bar{q} \left(\ddot{\gamma}_1 + L \omega(\cu) \cut \right)- \bar{p}\left( \ddot{\gamma}_2 - L\omega(\cu) \cuo \right)\right) E_1 + \\
&\quad \left( \bar{r}_L \bar{p} \left(\ddot{\gamma}_1 + L \omega(\cu) \cut \right)+ \bar{r}_L \bar{q}\left( \ddot{\gamma}_2 - L\omega(\cu) \cuo \right) - \dfrac{l}{l_L}\sqrt{L} \omega(\ddot{\gamma})\right) E_2.
\end{aligned}\end{equation*}
 and
\begin{equation*}
\begin{aligned}
 \cu &= \left(\bar{q} \cuo - \bar{p}\cut \right) E_1 +
\left( \bar{r}_L \bar{p} \cuo + \bar{r}_L \bar{q}\cut - \dfrac{l}{l_L}\sqrt{L} \omega(\cu)\right) E_2,
\end{aligned}\end{equation*}
 and therefore, by definition of $J_L$,
\begin{equation*}
J_L(\cu) = -\left( \bar{r}_L \bar{p} \cuo + \bar{r}_L \bar{q}\cut - \dfrac{l}{l_L}\sqrt{L} \omega(\cu)\right)E_1 + \left(\bar{q} \cuo - \bar{p}\cut \right) E_2.
\end{equation*}
After some simplifications, we get
\begin{equation*}
 \begin{aligned}
\scal{D_{t}^{\Sigma}\cu}{J_{L}(\cu)}_{L,\Sigma} &= \bar{r}_{L}\cuo \left( \ddot{\gamma}_2 - L \omega(\cu)\cut\right)(\bar{p}^{2}+\bar{q}^{2})
 -\bar{r}_{L}\cut \left( \ddot{\gamma}_1 + L \omega(\cu)\cuo \right)(\bar{p}^{2}+\bar{q}^{2})\\
&\quad  \dfrac{\sqrt{L} \, l}{l_L} \left[ \omega(\cu)\left( \bar{q}\ddot{\gamma}_1 + L\bar{q}\omega(\cu)\cut -\bar{p}\ddot{\gamma}_2
+ L \bar{p}\omega(\cu)\cuo\right) +\bar{q}\omega(\ddot{\gamma})\cuo - \bar{p}\omega(\ddot{\gamma})\cut \right]
 \end{aligned}
\end{equation*}
Therefore, exploiting Lemma \ref{limits}, for a non-horizontal curve $\gamma$, we have
$$\scal{D_{t}^{\Sigma}\cu}{J_{L}(\cu)}_{L,\Sigma} \sim L^{3/2} \omega(\cu)^{2} \left( \bar{q}\cut + \bar{p}\cuo \right), \quad \mbox{as $L \to +\infty$}.$$
Recalling that $\|\cu\|_{L,\Sigma} \sim \sqrt{L}|\omega(\cu)|$, as $L \to +\infty$, we find
$$\dfrac{\scal{D_{t}^{\Sigma}\cu}{J_{L}(\cu)}_{L,\Sigma}}{\|\cu\|_{L,\Sigma}^{3}} \to \dfrac{\left( \bar{p}\cuo+ \bar{q}\cut  \right)}{|\omega(\cu)|}, \quad \mbox{as $L \to +\infty$}.$$
As for the previous results, the case of $\gamma$ horizontal is
slightly easier, because $\omega(\cu) = \omega(\ddot{\gamma})=0$. Therefore,
\begin{align*}
D^{\Sigma}_t \cu &= \left( \bar{q} \left(\ddot{\gamma}_1 \right)- \bar{p}\left( \ddot{\gamma}_2 \right)\right) E_1 + \bar{r}_L \left( \bar{p} \left(\ddot{\gamma}_1 \right) +  \bar{q}\left( \ddot{\gamma}_2  \right)\right) E_2,\\
\cu &= \left(\bar{q} \cuo - \bar{p}\cut \right) E_1 + \bar{r}_L \left(  \bar{p} \cuo + \bar{q}\cut \right) E_2,\\
J_L(\cu)  &= -\bar{r}_L \left(  \bar{p} \cuo + \bar{q}\cut \right) E_1 +  \left(\bar{q} \cuo - \bar{p}\cut \right) E_2
\end{align*}
Hence, from Lemma \ref{limits},
$$\scal{D^{\Sigma}_t \cu}{J_L(\cu)}_L = \bar{r}_L \left[- \left(\bar{q} \cuo - \bar{p}\cut \right)\left(  \bar{p} \cuo + \bar{q}\cut \right) +
\left(  \bar{p} \cuo + \bar{q}\cut \right)\left(\bar{q} \cuo - \bar{p}\cut \right)\right] \to 0, \quad \mbox{as $L \to +\infty$},$$
 as desired.
\end{proof}

\begin{remark}
The sub-Riemannian geodesic curvature (both signed and unsigned) is invariant under isometries of $\heis$ (left translations and rotations about the $x_3$-axis), and scales by the factor $\tfrac1r$ under the dilation $\delta_r$. Compare Proposition \ref{curvature-isometry-proposition} and Remark \ref{curvature-dilation-remark}. We omit the elementary proofs of these facts.
\end{remark}


\section{Riemannian approximation of curvatures of surfaces}\label{C3}
In this section we want to study curvatures of regular surfaces $\Sigma \subset \heis$.
As in the previous sections, the idea will be to compute first the already
known curvatures of a regular 2-dimensional Riemannian manifold $\Sigma \subset (\reals^{3},g_{L})$,
and then try to derive appropriate Heisenberg notions taking the limit as $L \to +\infty$.

First, we need to define the \textit{second fundamental form} $II^{L}$ of the embedding of
$\Sigma$ into $(\reals^{3}, g_{L})$:
\begin{equation}\label{eq:secfun}
II^{L} = \left( \begin{array}{cc}
                   \scal{\nabla_{E_1}\nu_L}{E_1}_{L} & \scal{\nabla_{E_1}\nu_L}{E_2}_{L}\\
                   \scal{\nabla_{E_2}\nu_L}{E_1}_{L} & \scal{\nabla_{E_2}\nu_L}{E_2}_{L}\\
                  \end{array}\right).
\end{equation}
The explicit computation of the second fundamental form in our case can be found in \cite{CDPT}.
For sake of completeness, we recall here the complete statement.

\begin{theorem}[\cite{CDPT}, Theorem 4.3]
The second fundamental form $II^{L}$ of of the embedding of
$\Sigma$ into $(\reals^{3}, g_{L})$ is given by
\begin{equation}\label{eq:IIL}
II^{L} = \left( \begin{array}{cc}
\dfrac{l}{l_L}(X_{1}(\bar{p})+X_{2}(\bar{q})) & -\tfrac12 {\sqrt{L}} - \dfrac{l_L}{l}\scal{E_1}{\hgrad(\bar{r}_{L})}_{L}\\
-\tfrac12 {\sqrt{L}} - \dfrac{l_L}{l}\scal{E_1}{\hgrad(\bar{r}_{L})}_{L} & -\dfrac{l^{2}}{l_{L}^{2}}\scal{E_2}{\hgrad (\tfrac{r}{l})}_{L} + X_{3}^{L}(\bar{r}_{L})\\
\end{array}\right).
\end{equation}
\end{theorem}

The Riemannian mean curvature $\cH_{L}$ of $\Sigma$ is
$$
\cH_L := \trace(II^L),
$$
while the Riemannian Gaussian curvature $\cK_{L}$ is
$$
\cK_L := \ovl{\cK}_{L}(E_1,E_2) +\det(II^{L}).
$$
Here $\ovl{\cK}_{L}(E_1,E_2)$ denotes the sectional curvature of $(\reals^{3},g_L)$ in the plane generated by $E_1$ and $E_2$.
We need to spend a few words about the last definition, which is actually
a result known as the \textit{Gauss lemma}. The geometric meaning is that
the second fundamental form $II^{L}$ is the right object to measure
the discrepancy between the two sectional curvatures, of the ambient
manifold, and of the isometrically immersed one.

\begin{proposition}
The horizontal mean curvature $\cH_0$ of $\Sigma \subset \heis$ is
given by
\begin{equation}\label{eq:H0}
\cH_0 = \lim_{L \to +\infty} \cH_{L} = \mathrm{div}_{H}\left( \dfrac{\hgrad u}{\|\hgrad u\|_{\he}} \right),
\end{equation}
and the sub-Riemannian Gaussian curvature $\cK_0$ is given by
\begin{equation}\label{eq:K0}
\begin{aligned}
\cK_0 &=  \lim_{L \to +\infty} \cK_{L} \\
&= -\left( \dfrac{X_3 u}{\|\hgrad u\|_{\he}}\right)^{2}
-  \left( \dfrac{X_2 u}{\|\hgrad u\|_{\he}}\right) X_1 \left(\dfrac{X_3 u}{\|\hgrad u\|_{\he}}\right)
+ \left( \dfrac{X_1 u}{\|\hgrad u\|_{\he}} \right) X_2 \left(\dfrac{X_3 u}{\|\hgrad u\|_{\he}} \right).
\end{aligned}
\end{equation}
\end{proposition}

In \eqref{eq:H0} the expression $\mathrm{div}_H$ denotes the {\it horizontal divergence} of a horizontal vector field, which is defined as follows: for a horizontal vector field $V = a\, X_1 + b \, X_2$,
$$
\mathrm{div}_H(V) = X_1(a) + X_2(b).
$$

\begin{proof}
By definition,
$$
\cH_{L} = \mathrm{trace}(II^{L}) = \dfrac{l}{l_L}(X_{1}(\bar{p})+X_{2}(\bar{q})) -\dfrac{l^{2}}{l_{L}^{2}}\scal{E_2}{\hgrad (\tfrac{r}{l})}_{L} + X_{3}^{L}(\bar{r}_{L}).
$$
We recall that $r = \tfrac{X_{3}u}{\sqrt{L}}$ and therefore in the limit as $L \to +\infty$, we obtain
$$
\cH_0 = \lim_{L \to +\infty} \cH_L = X_1 \left( \dfrac{X_1 u}{\|\hgrad u\|_{\he}}\right) + X_2 \left( \dfrac{X_2 u}{\|\hgrad u\|_{\he}}\right).
$$
As we have already observed in the computation of $\cH_0$, the term
$$
-\dfrac{l^{2}}{l_{L}^{2}}\scal{E_2}{\hgrad (\tfrac{r}{l})}_{L} + X_{3}^{L}(\bar{r}_{L})
$$
tends to zero as $L \to +\infty$, and therefore in analyzing $\det(II^{L})$ we can focus only on the term
$$
- \left(-\dfrac{\sqrt{L}}{2} - \dfrac{l_L}{l}\scal{E_1}{\hgrad(\bar{r}_{L})}_{L}\right)^{2}.
$$
Clearly,
$$
- \left(-\dfrac{\sqrt{L}}{2} - \dfrac{l_L}{l}\scal{E_1}{\hgrad(\bar{r}_{L})}_{L}\right)^{2} \sim -\dfrac{L}{4} - \scal{E_1}{\hgrad \left(\dfrac{X_3 u}{\|\hgrad u\|_{\he}}\right)}_{L} \qquad \mbox{as $L \to +\infty$.}
$$
It remains to compute the sectional curvature $\ovl{\cK}_{L}(E_{1},E_{2})$. By Definition \ref{sectional}, the functional property in Remark \ref{Rfunct}, and orthonormality of the basis $\{E_{1}, E_{2}\}$, we have
$$
\ovl{\cK}_{L}(E_{1},E_{2}) = - \dfrac{L}{4} - L \, \bar{r}_{L}^{2}.
$$
Hence
 \begin{equation*}
 \begin{aligned}
 \cK_0 &=  \lim_{L \to +\infty} \cK_{L} \\
 &= -\left( \dfrac{X_3 u}{\|\hgrad u\|_{\he}}\right)^{2}
-  \left( \dfrac{X_2 u}{\|\hgrad u\|_{\he}}\right) X_1 \left(\dfrac{X_3 u}{\|\hgrad u\|_{\he}}\right)
 + \left( \dfrac{X_1 u}{\|\hgrad u\|_{\he}} \right) X_2 \left(\dfrac{X_3 u}{\|\hgrad u\|_{\he}} \right),
 \end{aligned}
 \end{equation*}
  as desired.
 \end{proof}

\begin{remark}\label{horiz-mean-curve-remark}
It is well known (see, for instance \cite{CDPT}) that the horizontal mean curvature of $\Sigma$ at a non-characteristic point $x$ coincides, up to a choice of sign, with the signed horizontal curvature of the Legendrian curve $\gamma$ in $\Sigma$ through $x$. This is the unique horizontal curve $\gamma$ defined locally near $x$, $\gamma:(-\eps,\eps) \to \Sigma$, such that $\gamma(0) = x$ and $\dot\gamma(0) = J(\nabla_H u/||\nabla_H u||_H) \in H_x\heis\cap T_x\Sigma$. The signed horizontal curvature of $\gamma=(\gamma_1,\gamma_2,\gamma_3)$ is given by
$$
\dfrac{\cuo \ddot{\gamma}_2 - \cut \ddot{\gamma}_1}{\sqrt{(\cuo^2 + \cut^2)^{3}}}.
$$
Observe that the absolute value of this expression coincides with the sub-Riemannian curvature $k_\gamma^0$ introduced in Definition \ref{def:sRc}, when considered on horizontal curves.
\end{remark}

The horizontal mean curvature $\cH_0$ and sub-Riemannian Gauss curvature $\cK_0$ have been given in terms of a defining function $u$ for the surface $\Sigma$. The preceding remark shows that $\cH_0$ is independent of the choice of the defining function, and the question arises whether such a result holds also for $\cK_0$. This is in fact the case, as we now demonstrate. Let $\Sigma$ be a $C^2$ surface defined locally near a point $x$ as the zero set of two $C^2$ functions $u$ and $v$ with $\nabla_{\reals^3} u\ne 0$ and $\nabla_{\reals^3} v \ne 0$. Then there exists a function $\sigma$ so that $v=e^\sigma u$ or $v=-e^\sigma u$ in a neighborhood of $x$. Without loss of generality we assume that $v=e^\sigma u$. The expressions
$$
\nu_0 = \frac{\nabla_H u}{||\nabla_H u||_H} \qquad \mbox{and} \qquad
\cP_0 = \frac{X_3u}{||\nabla_H u||_H}
$$
which occur in the expression for $\cK_0$ remain invariant under change of defining function. Indeed,
$$
\frac{\nabla_H v}{||\nabla_H v||_H} = \frac{e^\sigma(\nabla_H u + u \nabla_H \sigma)}{e^{\sigma}||\nabla_H u + u \nabla_H \sigma||_H}
\quad \mbox{and} \quad
\frac{X_3v}{||\nabla_H v||_H} = \frac{e^\sigma(X_3u+uX_3\sigma)}{e^\sigma||\nabla_H u + u \nabla_H \sigma||},
$$
which equal
$$
\frac{\nabla_H u}{||\nabla_H u||_H} \qquad \mbox{and} \qquad
\frac{X_3u}{||\nabla_H u||_H}
$$
respectively when evaluated at a point of $\Sigma$ (where $u=0$). Let us see how the horizontal gradient of $\cP_0$ transforms under such an operation. An easy computation gives
$$
\nabla_H \left( \frac{X_3v}{||\nabla_H v||_H} \right) = \nabla_H \left( \frac{X_3u}{||\nabla_H u||_H} \right) +
(X_3\sigma - \cP_0 \langle \nu_0,\nabla_H\sigma\rangle_H) \nu_0
$$
when evaluated on $\Sigma$. Since
$$
\cK_0 = -\cP_0^2 - \left\langle \nabla_H \cP_0,J\nu_0 \right\rangle_H
$$
we conclude that the expression for $\cK_0$ is independent of the choice of defining function.

We conclude this section with a brief discussion of the local summability of the horizontal
Gaussian curvature $\cK_0$ with respect to the Heisenberg
perimeter measure near isolated characteristic points. This observation will be used in our
subsequent study of the validity of a sub-Riemannian Gauss--Bonnet theorem in $\heis$.
Without loss of generality, suppose that the origin
is an isolated characteristic point of $\Sigma$. Consider
a neighborhood $U$ of the origin on $\Sigma$.
Due to the expression (\ref{eq:K0}), and since
the Heisenberg perimeter measure $d\sigma_{\he}$ is given by
$$d\sigma_{\he}=\dfrac{\|\hgrad u\|_{\he}}{\|\grad_{\reals^{3}} u\|_{\reals^{3}}}  d\mathcal{H}^{2}_{\mathbb{R}^{3}},$$
there exists a positive constant $C=C(U)>0$ such that
\begin{equation}\label{eq:summ}
|\cK_0| d\sigma_{\he} \leq \dfrac{C}{\|\hgrad u\|_\he} d\mathcal{H}^{2}_{\mathbb{R}^{3}}.
\end{equation}
Finer results turn out to be particularly difficult to achieve. We postpone further discussion of this topic to Section \ref{questions}.

\section{A Gauss--Bonnet theorem}\label{gb}

The goal of this section is to prove Theorem \ref{HGB}.
Before doing that, we need to recall a
couple of technical results concerning, respectively
the Riemannian length measure, and the Riemannian surface
measure.
Let us first consider the case of a curve $\gamma:[a,b] \to (\mathbb{R}^{3},g_L)$
in the Riemannian manifold $(\mathbb{R}^{3},g_L)$.
We define the Riemannian length measure,
$$
d\cu_{L} := \gamma_{\sharp}\left( \|\cu\|_L \, d\mathcal{L}^{1}_{\llcorner [a,b]}\right) = \|\cu\|_L \, dt.
$$

\begin{lemma}\label{Lmeas}
Let $\gamma:[a,b] \to (\mathbb{R}^{3},g_L)$ be a Euclidean $C^{2}$-smooth and regular curve.
Then
\begin{equation}\label{eq:curvemeas}
\lim_{L \to +\infty} \dfrac{1}{\sqrt{L}}\int_{\gamma} d\cu_{L} =
\int_{a}^{b}|\omega(\cu)| \, dt =: \int_{\gamma} \, d\dot{\gamma},
\end{equation}
\end{lemma}

\begin{proof}
As we already saw, $\|\cu\|_L = \sqrt{\cuo^{2}+\cut^{2}+L \,\omega(\cu)^{2}}$, hence by the definition of $d\cu_{L}$ and the dominated convergence theorem,
\begin{equation*}
  \lim_{L \to +\infty} \dfrac{1}{\sqrt{L}}\int_{\gamma} d\cu_{L} = \int_{a}^{b}\lim_{L \to +\infty} \dfrac{\|\cu\|_L}{\sqrt{L}}\, dt
= \int_{a}^{b}\lim_{L \to +\infty} \dfrac{\sqrt{\cuo^{2}+\cut^{2}+L \,\omega(\cu)^{2}}}{\sqrt{L}}\, dt
= \int_{a}^{b} |\omega(\cu)| \, dt
\end{equation*}
 as desired.
\end{proof}

\begin{remark}
It is clear that this scaled measure vanishes on fully horizontal curves.
\end{remark}

Let us also recall a technical result concerning the scaled limit of the Riemannian surface measure.

\begin{proposition}[\cite{CDPT}, Chapter 5.1.]\label{LSmeas}
Let $\Sigma \subset (\mathbb{R}^{3},g_L)$ be a Euclidean $C^{2}$-smooth surface.
Let $d\sigma_{L}$ denote the surface measure on $\Sigma$ with respect to the Riemannian metric $g_L$.
Let $M$ be the $2 \times 3$ matrix
\begin{equation}\label{eq:M}
 M:= \left (
\begin{array}{ccc}
1 & 0 & -\tfrac{x_2}{2} \\
0 & 1 & \tfrac{x_1}{2} \\
\end{array}
\right ).
\end{equation}
If $\Sigma = \{u=0\}$  with $u \in C^{2}(\mathbb{R}^{3})$, then
\begin{equation}\label{eq:SMeas1}
\lim_{L \to \infty} \dfrac{1}{\sqrt{L}} \int_{\Sigma} d\sigma_{L} = \int_{\Sigma}\dfrac{\|\hgrad u\|_{\he}}{\|\nabla_{\mathbb{R}^{3}}u\|_{\reals^{3}}} \, d\mathcal{H}^{2}_{\mathbb{R}^{3}} = \int_{\Sigma} d\mathcal{H}^{3}_{cc},
\end{equation}
 where $d\mathcal{H}^{2}_{\mathbb{R}^{3}}$ denotes the Euclidean 2-Hausdorff measure and
 $d\mathcal{H}^{3}_{cc}$ the 3-dimensional Hausdorff measure with respect to the $cc$ metric $d_{cc}$.
If $\Sigma = f(D)$ with
$$
f = f(\boldu,\boldv):D\subset \mathbb{R}^{2} \to \mathbb{R}^{3},$$
 and Euclidean normal vector to $\Sigma$ given by $\vec{n}(\boldu,\boldv)= (f_{\boldu} \times f_{\boldv})(\boldu,\boldv)$, then
\begin{equation}\label{eq:SMeas2}
\lim_{L \to \infty} \dfrac{1}{\sqrt{L}} \int_{\Sigma} d\sigma_{L} =  \int_{D} \|M\vec{n}\|_{\mathbb{R}^{2}} \, d\boldu d\boldv.
\end{equation}
\end{proposition}

The classical Gauss--Bonnet theorem for a regular surface $\Sigma \subset (\mathbb{R}^{3},g_L)$
with boundary components given by Euclidean $C^{2}$-smooth and regular curves $\gamma_{i}$
(see for instance \cite[Chapter 9]{L} or \cite{dC}) states that
\begin{equation}\label{eq:GB}
\int_{\Sigma} \cK_{L}\, d\sigma_{L}
 + \sum_{i=1}^{n}\int_{\gamma_i} k_{\gamma_i, \Sigma}^{L,s} \, d\dot{\gamma_i}_{L}
= 2 \pi \chi(\Sigma),
\end{equation}
where $\cK_{L}$ is the Gaussian curvature of $\Sigma$, $k_{\gamma_i, \Sigma}^{L,s}$
is the signed geodesic curvature of the $i^{th}$
boundary component $\gamma_i$, $d\dot{\gamma_i}_{L}  = \|\cu_i\|_{L} d\theta$
and $\chi(\Sigma)$ is the Euler characteristic of $\Sigma$. It is clear that for a regular
surface $\Sigma \subset (\reals^{3}, g_{L})$ without boundary, (\ref{eq:GB}) simplifies to
$$\int_{\Sigma}\cK_{L} \,d\sigma_{L} = 2\pi \chi(\Sigma).$$
 Recalling the considerations made on the Riemannian surface measure $d\sigma_{L}$,
it is natural to divide (\ref{eq:GB}) by a factor $\sqrt{L}$,
\begin{equation}\label{eq:GBScal}
\int_{\Sigma} \cK_{L}\, \dfrac{d\sigma_{L}}{\sqrt{L}}
 + \sum_{i=1}^{n}\int_{\gamma_i} k_{\gamma_i, \Sigma}^{L,s} \, \dfrac{d\dot{\gamma_i}_{L}}{\sqrt{L}}
= \dfrac{2 \pi \chi(\Sigma)}{\sqrt{L}} ,
\end{equation}
and then hope to derive a Gauss--Bonnet Theorem as a limit as $L \to +\infty$.
The most difficult task is to take care of the possible presence
of characteristic points on $\Sigma$. In order to deal with them,
we provide the following general definition:

\begin{definition}[(R)-property]
Let $S \subset \mathbb{R}^{2}$ be any set in $\mathbb{R}^{2}$.
We say that the set $S$ satisfies the removability
(R)-property, if for every $\ep >0$,
there exist a number $n = n(\ep) \in \mathbb{N}$
and smooth simple closed curves
$\gamma_{1}, \ldots, \gamma_n: I\subset \mathbb{R} \to \mathbb{R}^{2}$,
such that
$$S \subset \bigcup_{i=1}^{n}\mathrm{int}(\gamma_i),
\quad \mbox{and} \quad \sum_{i=1}^{n} \mathrm{length}(\gamma_i) \leq \ep,$$
 where $\mathrm{length}_{E}(\gamma)$ denotes the usual
Euclidean length of a curve in $\mathbb{R}^{2}$.
\end{definition}

It is clear that if the set $S$ consists of finitely many isolated points, then it satisfies
the (R)-property. A more complicated example is provided
by the self-similar Cantor set $C^{(2)}(\lambda)$ in $\mathbb{R}^{2}$ with scaling ratio $\lambda< 1/4$. Here $C^{(2)}(\lambda) = C^{(1)}(\lambda) \times C^{(1)}(\lambda)$, where $C^{(1)}(\lambda)$ denotes the unique nonempty compact subset of $\reals$ which is fully invariant under the action of the two contractive similarities $f_1(x) = \lambda x$ and $f_2(x) = \lambda x + 1 - \lambda$.

\begin{proposition}
When $\lambda < \tfrac14$, the self-similar Cantor set $C^{(2)}(\lambda)$ satisfies the (R)-property.
\end{proposition}

\begin{proof}
At the $n^{th}$ stage of the iterative construction of $C^{(2)}(\lambda)$
we have $4^{n}$ pieces. We can surround
every such piece with a smooth closed curve $\gamma_{n}^{k}$. For sake of
simplicity let us take a Euclidean
circle whose radius $r_{n}^{k}$ is comparable to $\lambda^n$. Therefore
\begin{equation*}
\sum_{k=1}^{4^{n}} \mathrm{length}_{E}(\gamma_{n}^{k})
\lesssim 4^{n} \lambda^{n} \to 0, \quad \mbox{as $n \to +\infty$},
\end{equation*}
because $\lambda < 1/4$.
\end{proof}

\begin{lemma}\label{Rprop}
If $S \subset \mathbb{R}^{2}$ is compact and such that $\mathcal{H}_{E}^{1}(S)=0$ then
$S$ satisfies the (R)-property.
\end{lemma}

\begin{proof}
First, since $S$ is compact,
we can take a finite covering of $S$
made of Euclidean balls $\{B(g_i, r_i)\}_{i=1,\ldots, m}$,
with $g_{i} \in S$ for every
$i= 1 ,\ldots,m$, and such that $\sum_{i=1}^{m}r_{i} \leq \ep$.
Enlarging these balls by a factor which will not depend on $\ep$, we
can assume that $S$ lies entirely in the interior of the union of
these balls. Now we define the surrounding curves,
as the boundaries of the unions of these
balls. By construction, the Euclidean length
of such curves is comparable to $\ep$.
\end{proof}

\begin{remark}
The converse of Lemma \ref{Rprop} also holds true.
In other words, the validity of the (R)-property for a compact set
$S\subset \mathbb{R}^{2}$ is equivalent to $\mathcal{H}_{E}^{1}(S)=0$.
\end{remark}

Before proving Theorem \ref{HGB}, let us make another useful remark.
Let us denote by $\Pi$ the projection
$\Pi : \heis \to \mathbb{R}^{2}$ onto the first two components. Consider
a surface $\Sigma \subset \heis$ as in the hypothesis of
Theorem \ref{HGB}. Then
$$
\mathcal{H}_{E}^{1}(C(\Sigma)) = \mathcal{H}_{E}^{1}(\Pi(C(\Sigma))).
$$
In particular, if $\mathcal{H}_{E}^{1}(C(\Sigma))=0$, then
its projection $\Pi(C(\Sigma))$ satisfies the (R)-property.

\begin{proof}[Proof of Theorem \ref{HGB}]
The proof of Theorem \ref{HGB} will be a combination of  different steps.\\

\textbf{Step 1}: First we consider the case of a regular surface without characteristic points. Precisely,
let $\Sigma \subset \heis$ be a regular surface
without characteristic points, and with
finitely many boundary components
$(\partial \Sigma)_i$, $i \in \{1, \ldots, n\}$,
given by Euclidean $C^{2}$-smooth closed curves
$\gamma_i : [0, 2\pi] \rightarrow (\partial \Sigma)_i$.
We may assume that none of the boundary components are fully horizontal.
Let $\cK_0$ be the sub-Riemannian Gaussian curvature of $\Sigma$,
and let $k_{\gamma_i,\Sigma}^{0,s}$ be
the sub-Riemannian signed geodesic curvature
of $\gamma_i$, for every $i \in \{1, \ldots, n\}$.
Then
$$\int_{\Sigma}\cK_0 \, d\mathcal{H}^{3}_{cc}
+\sum_{i=1}^{n} \int_{\gamma_i} k_{\gamma_i,\Sigma}^{0,s} \, d\dot{\gamma}_i = 0.$$
\begin{proof}[Proof of Step 1]
The results will follow passing to the limit as $L \to +\infty$ in
\begin{equation*}
\int_{\Sigma} \cK_{L}\, \dfrac{d\sigma_{L}}{\sqrt{L}}
 + \sum_{i=1}^{n}\int_{\gamma_i} k_{\gamma_i, \Sigma}^{L,s} \, \dfrac{d\dot{\gamma_i}_{L}}{\sqrt{L}}
= \dfrac{2 \pi \chi(\Sigma)}{\sqrt{L}}.
\end{equation*}
To do this, we need to apply Lebesgue's dominated convergence theorem.
Let us start with the integral of $\cK_L$.
We take a partition of unity $\{\varphi_{i}\}$, $i=1,\ldots, m$.
Calling $\Sigma_{i} := \textrm{supp}(\varphi_{i}) \cap \Sigma$, for every $i=1,\ldots,m$, we
have
$$\int_{\Sigma} \cK_{L}\, \dfrac{d\sigma_{L}}{\sqrt{L}}=
\sum_{i=1}^{m}\int_{\Sigma_i}\cK_{L} \varphi_{i}\, \dfrac{d\sigma_{L}}{\sqrt{L}}.$$
 Let us choose a parametrization of every $\Sigma_{i}$, $\psi_{i}:D_{i} \to \Sigma_{i}$,
then, for every $i=1, \ldots,m$, it holds that
$$\int_{\Sigma_i}\cK_{L}\, \varphi_{i}\, \dfrac{d\sigma_{L}}{\sqrt{L}} =
\int_{D_i}\cK_{L} \, \varphi_{i} \, |M_{L} \vec{n}|\, dv \, dw, $$
 where
\begin{equation*}
 M_{L}:= \left (
\begin{array}{ccc}
1 & 0 & -\tfrac{x_2}{2} \\
0 & 1 & \tfrac{x_1}{2} \\
0 & 0 & \tfrac{1}{\sqrt{L}}\\
\end{array}
\right),
\end{equation*}
 and $\vec{n}$ denotes the Euclidean normal vector to $\Sigma$.\\
 It is now sufficient to check
whether there exist two positive constants $M_1$ and $M_2$, independent on $L$,
such that
$$|\cK_{L}| \leq M_1, \quad \textrm{and} \quad |M_{L} \vec{n}|\leq M_2.$$
 The second estimate is proved in Proposition \ref{LSmeas}, see \cite{CDPT}.
For the first one, we recall that the explicit expression of $\cK_L$ is given by
\begin{equation}
\begin{aligned}
\cK_L & = -L \bar{r}_{L}^{2} - \left( \dfrac{l}{l_L}\right)^{3} (X_1 \bar{p} + X_2 \bar{q}) \scal{E_2}{\hgrad \left( \dfrac{r}{l}\right)}_{L} + \dfrac{l}{l_L}(X_1 \bar{p} + X_2 \bar{q}) X_{3}^{L}\bar{r}_L\\
& - \left( \dfrac{l}{l_L}\right)^{4} \left( \scal{E_1}{\hgrad \left( \dfrac{r}{l}\right)}_{L} \right)^{2} - \sqrt{L}\left( \dfrac{l}{l_L}\right)^{2}\scal{E_1}{\hgrad \left( \dfrac{r}{l}\right)}_{L}.
\end{aligned}
\end{equation}
 Since
$l_{L} = \|\nabla_{L}u \|_{L} = \sqrt{(X_1 u)^{2}+(X_2 u)^{2}+ \left( \dfrac{X_3 u}{\sqrt{L}} \right)^{2}}
\geq \|\hgrad u\|_{\he},$
we have that
$$
\dfrac{l}{l_L} = \dfrac{\|\hgrad u\|_{\he}}{\| \nabla_{L}u\|_{L}} \leq 1.
$$
Since $L>1$, it also holds that
$$l_L \leq \sqrt{(X_1 u)^{2}+(X_2 u)^{2}+(X_3 u)^{2}}.$$
 Moreover, since $u \in C^{2}(\reals^{3})$ and
$\Sigma$ is a compact surface without characteristic points,
there exists a positive constant $C_1 >0$ such that
\begin{equation*}
|X_1 \bar{p} + X_2 \bar{q}| = \left| X_1 \left( \dfrac{X_{1}u}{\|\hgrad u\|_{\he}}\right) + X_2 \left( \dfrac{X_{2}u}{\|\hgrad u\|_{\he}}\right)\right| \leq C_1.
\end{equation*}

Therefore we have the following list of estimates:
\begin{equation*}
  L \bar{r}_{L}^{2} = L \dfrac{(X_{3}^{L}u)^{2}}{l_{L}^{2}} \leq \left( \dfrac{X_3 u}{\|\hgrad u\|_{\he}}\right)^{2} \leq C_2.
\end{equation*}
\begin{equation*}
 \begin{aligned}
\left| \scal{E_2}{\hgrad \left( \dfrac{r}{l}\right)}_{L} \right| &= \left| \bar{r}_L \bar{p} X_1 \left(\dfrac{r}{l}\right) +\bar{r}_L \bar{q} X_2 \left(\dfrac{r}{l}\right) \right| \\
&= \dfrac{|X_3 u|}{L} \left| \dfrac{X_1 u}{\|\hgrad u\|_{\he}} X_1 \left( \dfrac{X_3 u}{\|\hgrad u\|_{\he}} \right) + \dfrac{X_2 u}{\|\hgrad u\|_{\he}} X_2 \left( \dfrac{X_3 u}{\|\hgrad u\|_{\he}} \right) \right|\\
& \leq |X_3 u| \,  \left| \dfrac{X_1 u}{\|\hgrad u\|_{\he}} X_1 \left( \dfrac{X_3 u}{\|\hgrad u\|_{\he}} \right) + \dfrac{X_2 u}{\|\hgrad u\|_{\he}} X_2 \left( \dfrac{X_3 u}{\|\hgrad u\|_{\he}} \right) \right| \leq C_3.
 \end{aligned}\end{equation*}
\begin{equation*}
 \begin{aligned}
  \sqrt{L} \left| \scal{E_1}{\hgrad \left( \dfrac{r}{l}\right)}_{L} \right| & = \sqrt{L} \left| \bar{q} X_{1}\left( \dfrac{X_3 u}{\sqrt{L}\|\hgrad u\|_{\he}}\right)
  -\bar{p} X_{2}\left( \dfrac{X_3 u}{\sqrt{L}\|\hgrad u\|_{\he}} \right)\right|\\
  &\leq \left| \dfrac{X_2 u}{\|\hgrad u\|_{\he}}  X_{1}\left( \dfrac{X_3 u}{\|\hgrad u\|_{\he}}\right)\right| +\left| \dfrac{X_1 u}{\|\hgrad u\|_{\he}}  X_{2}\left( \dfrac{X_3 u}{\|\hgrad u\|_{\he}}\right)\right| \leq C_4.
 \end{aligned}\end{equation*}
Similarly,
$$|X_{3}^{L}\bar{r}_{L}| \leq C_5.$$
Altogether,
$$|\cK_{L}|\leq C_2 + C_1 \cdot C_3 + C_1 \cdot C_5 + C_{4}^{2} + C_4 =: M_1,$$
as desired.

 It remains to see what happens for the boundary integrals.
Without loss of generality, let us assume we are given
a surface $\Sigma$ with only one boundary component,
given by a smooth curve $\gamma$.
We need to estimate
$$\left| \dfrac{\scal{D_{t}^{\Sigma}\dot{\gamma}}{J_{L}(\dot{\gamma})}_{L}}{\sqrt{L} \|\dot{\gamma}\|_{L}^{2}}\right|,$$
 where
\begin{equation*}
 \begin{aligned}
 \scal{D_{t}^{\Sigma}\dot{\gamma}}{J_{L}(\dot{\gamma})}_{L} & = \bar{r}_L (\dot{\gamma}_1 \ddot{\gamma}_2 - \dot{\gamma}_2 \ddot{\gamma}_1) \\
 &+ \sqrt{L} \dfrac{l}{l_L} \left[ \omega(\dot{\gamma}) \left( \bar{q}\ddot{\gamma}_1 - \bar{p}\ddot{\gamma}_2\right) + \omega(\ddot{\gamma}) \left( \bar{q}\dot{\gamma}_1 - \bar{p}\dot{\gamma}_2\right) \right]\\
 &- L \omega(\dot{\gamma}) \bar{r}_L (2\dot{\gamma}_1 \dot{\gamma}_2) + L^{3/2} \dfrac{l}{l_L} \omega(\dot{\gamma})^{2}(\bar{p} \dot{\gamma}_1 + \bar{q}\dot{\gamma}_2),
 \end{aligned}\end{equation*}
 and
$$\|\dot{\gamma}\|_{L}^{2} = \dot{\gamma}_{1}^{2}+ \dot{\gamma}_{2}^{2} + L \omega(\dot{\gamma})^{2}.$$
Note that there exists a positive constant $C_0 >0$, such that
$$\dot{\gamma}_{1}^{2}+ \dot{\gamma}_{2}^{2} + L \omega(\dot{\gamma})^{2} \geq \dot{\gamma}_{1}^{2}+ \dot{\gamma}_{2}^{2} + \omega(\dot{\gamma})^{2} \geq C_0 >0.$$
Now, we have
\begin{equation*}
 \begin{aligned}
  \dfrac{|\bar{r}_L (\dot{\gamma}_1 \ddot{\gamma}_2 - \dot{\gamma}_2 \ddot{\gamma}_1)|}{\sqrt{L}(\dot{\gamma}_{1}^{2}+ \dot{\gamma}_{2}^{2} + L \omega(\dot{\gamma})^{2})}
  & \leq \dfrac{|\dot{\gamma}_1 \ddot{\gamma}_2 - \dot{\gamma}_2 \ddot{\gamma}_1||X_3 u|}{L \|\hgrad u\|_{\he} (\dot{\gamma}_{1}^{2}+ \dot{\gamma}_{2}^{2} + L \omega(\dot{\gamma})^{2})}\\
  &\leq \dfrac{|X_3 u|}{\|\hgrad u\|_{\he}} \, \dfrac{|\dot{\gamma}_1 \ddot{\gamma}_2 - \dot{\gamma}_2 \ddot{\gamma}_1|}{\dot{\gamma}_{1}^{2}+ \dot{\gamma}_{2}^{2} + \omega(\dot{\gamma})^{2}}\leq C_6.
 \end{aligned}\end{equation*}
\begin{equation*}
  \sqrt{L} \dfrac{l}{l_L} \dfrac{|\left[ \omega(\dot{\gamma}) \left( \bar{q}\ddot{\gamma}_1 - \bar{p}\ddot{\gamma}_2\right) + \omega(\ddot{\gamma}) \left( \bar{q}\dot{\gamma}_1 - \bar{p}\dot{\gamma}_2\right) \right]|}{\sqrt{L}(\dot{\gamma}_{1}^{2}+ \dot{\gamma}_{2}^{2} + L \omega(\dot{\gamma})^{2})}
  \leq \dfrac{|\omega(\dot{\gamma})| |\bar{q} \ddot{\gamma}_2 - \bar{p} \ddot{\gamma}_2| + |\omega(\ddot{\gamma})| |\bar{q}\dot{\gamma}_1 - \bar{p}\dot{\gamma}_2 |}{\dot{\gamma}_{1}^{2}+ \dot{\gamma}_{2}^{2} + \omega(\dot{\gamma})^{2}}\leq C_7.
\end{equation*}
\begin{equation*}
 \dfrac{2 L |\omega(\dot{\gamma})| |\bar{r}_L| |\dot{\gamma}_1 \dot{\gamma}_2|}{\sqrt{L}(\dot{\gamma}_{1}^{2}+ \dot{\gamma}_{2}^{2} + L \omega(\dot{\gamma})^{2})}
 \leq \dfrac{2 |\omega(\dot{\gamma})| |X_3 u| |\dot{\gamma}_1 \dot{\gamma}_2|}{\|\hgrad u\|_{\he}\, (\dot{\gamma}_{1}^{2}+ \dot{\gamma}_{2}^{2} + \omega(\dot{\gamma})^{2})}\leq C_8.
\end{equation*}
\begin{equation*}
  \dfrac{l}{l_L} \, \dfrac{L^{3/2} \omega(\dot{\gamma})^{2} |\bar{p}\dot{\gamma}_1 + \bar{q}\dot{\gamma}_2|}{\sqrt{L}(\dot{\gamma}_{1}^{2}+ \dot{\gamma}_{2}^{2} + L \omega(\dot{\gamma})^{2})}
 \leq \dfrac{L^{3/2} \omega(\dot{\gamma})^{2}}{L^{3/2} \omega(\dot{\gamma})^{2}} |\bar{p}\dot{\gamma}_1 + \bar{q}\dot{\gamma}_2|
 = |\bar{p}\dot{\gamma}_1 + \bar{q}\dot{\gamma}_2|\leq C_9.
\end{equation*}
The behavior of the measure has been already treated in Lemma \ref{Lmeas}.
\end{proof}

\textbf{Step 2}: Due to Lemma \ref{Rprop}, we can surround the projection of
the characteristic
set $\Pi(C(\Sigma))$ with smooth simple closed curves
$\{\beta_{j}\}_{j=1,\ldots,n(\ep)}$
such that
\begin{equation}\label{eq:Rproof}
\sum_{j=1}^{n(\ep)}\mathrm{length}_{E}(\beta_{j}) \leq \ep.
\end{equation}
 We can now work with a new surface $\Sigma_{\ep}$
which has no characteristic
points, and boundary components which are given by the curves
$\gamma_{i}$'s and the curves $\beta_{j}$'s. Step 1
tells us that for every $\ep >0$,
\begin{equation*}
\int_{\Sigma_{\ep}}\cK_0 \, d\mathcal{H}_{cc}^{3} +
\sum_{i=1}^{n}\int_{\gamma_{i}}k_{\gamma_{i}, \Sigma_{\ep}}^{0,s} \, d\dot{\gamma}_{i}
= -\sum_{j=1}^{n(\ep)}k_{\beta_{j}, \Sigma_{\ep}}^{0,s} \, d\dot{\beta}_{j} ,
\end{equation*}
which, combined with (\ref{eq:Rproof}), implies that for every $\ep >0$,
\begin{equation*}
\left| \int_{\Sigma_{\ep}} \cK_0 \, d\mathcal{H}_{cc}^{3}+
\sum_{i=1}^{n}\int_{\gamma_{i}}k_{\gamma_{i}, \Sigma_{\ep}}^{0,s} \, d\dot{\gamma}_{i}\right| =
\left| \sum_{j=1}^{n(\ep)}k_{\beta_{j}, \Sigma_{\ep}}^{0,s} \, d\dot{\beta}_{j} \right| \leq \ep,
\end{equation*}
 and this completes the proof.
\end{proof}

\begin{corollary}
Let $\Sigma\subset \heis$ be a regular surface without boundary, or
with boundary components given by Euclidean $C^{2}$-smooth horizontal curves.
Assume that the characteristic set $C(\Sigma)$ satisfies $\mathcal{H}^{1}(C(\Sigma))=0$
and that $\|\hgrad u\|_{\he}^{-1}$ is locally summable with respect to
the Euclidean 2-dimensional Hausdorff measure, near the characteristic set $C(\Sigma)$.
Then the sub-Riemannian Gaussian curvature $\cK_0$ cannot be an always positive or negative function.
In particular, $\Sigma$ cannot have constant non-zero sub-Riemannian Gaussian curvature $\cK_0$.
\end{corollary}

\begin{proof}
If $\Sigma$ has no boundary, then by Theorem \ref{HGB} we have
$$\int_{\Sigma}\cK_0\, d\mathcal{H}_{cc}^{3}=0,$$
and therefore $\cK_0$ cannot have a sign.\\
The same holds for boundary components given by
horizontal curves because the sub-Riemannian signed curvature
$k_{\gamma,\Sigma}^{0,s}$ of
a horizontal curve $\gamma$ is $0$.
\end{proof}

\begin{remark}
At the moment we do not know any example of a regular surface with
sub-Riemannian Gaussian curvature $\cK_0$ constantly equal to zero.
On the other hand, if we remove the requirement of $\Sigma$ being
compact, all vertically ruled (smooth) surfaces have $\cK_0=0$ at every point.
\end{remark}

The following examples show that the assumption
made in Theorem \ref{HGB} about
the 1-dimensional Euclidean Hausdorff measure
of the characteristic set $C(\Sigma)$ is sharp.

\begin{example}
 Let $\Sigma = \{ (x_1,x_2,x_3)\in \heis: u(x_1,x_2,x_3)=0 \}$, with
\begin{equation}
u(x_1,x_2,x_3)=
 \left \{
 \begin{array}{rl}
 x_3 - \dfrac{x_1 x_2}{2} +  x_2\, \mathrm{exp}((x_1+1)^{-2}), & x_1 < 1, x_2,x_3 \in \mathbb{R}, \\
 x_3 - \dfrac{x_1 x_2}{2}, & x_1 \in [-1, 1], x_2,x_3 \in \mathbb{R}, \\
 x_3 - \dfrac{x_1 x_2}{2} +  x_2\, \mathrm{exp}((x_1-1)^{-2}), & x_1 > 1, x_2,x_3 \in \mathbb{R} .
 \end{array}\right.
 \end{equation}
We have that
$$C(\Sigma) = \{ (x_1,0,0): x_1\in [-1,1]\}.$$
The idea now is to consider the projection of $C(\Sigma)$ and to surround it
with a curve in $\mathbb{R}^{2}$. Then, we will lift it to the surface $\Sigma$,
exploiting that $\Sigma$ is given by a graph.
For $\ep >0$, define
$$\gamma(\theta) := \bigcup_{i=1}^{5}\gamma_{i}(\theta),\quad \mbox{$\theta \in [0,2\pi]$},$$
where
$$\gamma_{1}(\theta):= \left( 1+ \ep \cos\left(\dfrac{\pi}{2 \tan(\ep)} \theta\right), \ep \sin\left(\dfrac{\pi}{2 \tan(\ep)} \theta\right)\right), \quad \mbox{$\theta \in [0,\tan(\ep))$},$$
$$\gamma_{2}(\theta):= \left(-\dfrac{2}{\pi-2 \tan(\ep)}\theta + \dfrac{\pi}{\pi - 2\tan(\ep)}, \ep \right), \quad \mbox{$\theta \in [\tan(\ep), \pi - \tan(\ep))$},$$
$$\gamma_{3}(\theta):= \left( -1+ \ep \cos\left(\alpha(\theta)\right), \ep \sin\left(\alpha(\theta)\right)\right), \quad \mbox{$\theta \in [\pi - \tan(\ep),\pi+\tan(\ep))$},$$
$$\gamma_{4}(\theta):= \left(t(\theta), -\ep \right), \quad \mbox{$\theta \in [\pi+\tan(\ep), 2\pi - \tan(\ep))$},$$
$$\gamma_{5}(\theta):= \left( 1+ \ep \cos\left(\beta(\theta)\right), \ep \sin\left(\beta(\theta)\right)\right), \quad \mbox{$\theta \in [2\pi - \tan(\ep),2\pi)$},$$
 with
$$\alpha(\theta):= \dfrac{\pi}{2 \tan(\ep)} \theta + \pi - \dfrac{\pi^{2}}{2\tan(\ep)},$$
$$t(\theta):= \dfrac{2}{\pi - 2 \tan(\ep)}\theta - \dfrac{3 \pi}{\pi - 2 \tan(\ep)},$$
$$\beta(\theta):= \dfrac{\pi}{\pi - 2\tan(\ep)}\theta + 2\pi \left( 1 - \dfrac{\pi}{2 \tan(\ep)}\right).$$
Now we need to control the integral of the signed curvature.
It will be made of five pieces. Three of them (namely for $i=1,3,5$), will behave exactly like when we
deal with isolated characteristic points, because the velocity of those parts goes to $0$ as $\ep \to 0$.
It remains to check the other two integrals. We can re-parametrize those two parts as follows:
$$\gamma_{2}(s) = ( -s, \ep), \quad \mbox{and} \quad \gamma_{4}(s)=(s,-\ep), \quad s \in [-1,1].$$
Because of the results concerning the signed curvature, it does not matter what happens to the
third component. In this situation, we get that
$$\dot{\gamma}_{2}(s) = (-1,0), \quad \mbox{and} \quad \dot{\gamma}_{4}(s)=(1,0), \quad s \in [-1,1],$$
 and
$$\bar{p}(x_1,x_2,x_3) = -\dfrac{x_2}{|x_2|}, \quad \quad \bar{q}(x_1,x_2,x_3) = 0.$$
Therefore $\bar{p}|_{\gamma_{2}} = 1$ and $\bar{p}|_{\gamma_{4}} = -1$.
Then
$$\int_{-1}^{1} \bar{p}|_{\gamma_2} (\dot{\gamma}_2)_1 \, ds + \int_{-1}^{1} \bar{p}|_{\gamma_4} (\dot{\gamma}_4)_1 \, ds = -2.$$
\end{example}

\begin{example}
Let $\Sigma = \{(x_1,x_2,x_3)\in \heis: x_3 = \tfrac{x_1 \, x_2}{2}\}$.
The projection of the characteristic set $\Pi(C(\Sigma))$ is the 1-dimensional
line $\{x_2 =0\}$. Consider the curve
$$\gamma(t) = \left( \cos(t), \sin(t), \dfrac{\sin(2t)}{4}\right), \quad t\in [0,2\pi],$$
which lives on the surface $\Sigma$ as boundary component of a
new surface $\tilde{\Sigma}$ which is now bounded. Simple computations
show that in this case
$$\cK_0 = 0, \quad k_{\gamma, \tilde{\Sigma}}^{0,s} = \dfrac{1}{|\sin(t)|},
\quad \omega(\dot{\gamma})= \dfrac{\cos(2t)}{2} - \dfrac{1}{2}.$$
Therefore,
$$\int_{\tilde{\Sigma}}\cK_0 \, d\mathcal{H}_{cc}^{3} + \int_{\gamma}k_{\gamma, \tilde{\Sigma}}^{0,s} \, d\dot{\gamma} = 4,$$
in contrast with the statement of Theorem \ref{HGB}.
\end{example}

We end this section with an explicit example.

\begin{example}[Kor\'{a}nyi sphere]
Consider the Kor\'{a}nyi sphere $\mathbb{S}_{\heis}$,
$$ \mathbb{S}_{\heis} := \{ (x_1,x_2,x_3) \in \heis : (x_{1}^2 + x_{2}^2)^2 + 16 x_{3}^2 -1 =0 \}.$$
A parametrization of  $\mathbb{S}_{\heis}$ is
$$f(\varphi,\theta) :=
\left( \sqrt{\cos(\varphi)} \cos(\theta), \sqrt{\cos(\varphi)} \sin(\theta) , \dfrac{\sin(\varphi)}{4}\right),
\quad \mbox{$\varphi \in \left(-\dfrac{\pi}{2},\dfrac{\pi}{2}\right)$, $\theta \in [0,2\pi)$}.$$
In particular, recalling the definition (\ref{eq:M}) of $M$ and after some computations, we get that
$$\| M (f_{\varphi} \times f_{\theta})\|_{\mathbb{R}^{2}} = \dfrac{\sqrt{\cos(\varphi)}}{4}.$$
A direct computation shows that the sub-Riemannian Gaussian curvature of $\mathbb{S}_{\heis}$ is
given by
\begin{equation}\label{eq:K0Kor}
\cK_0 = -\dfrac{2}{x_{1}^2 + x_{2}^2} + 6 (x_{1}^2 + x_{2}^2) = -\dfrac{2}{\cos(\varphi)} + 6 \cos(\varphi),
\end{equation}
which is locally summable around the isolated characteristic points
with respect to the the Heisenberg perimeter measure.

Thus, by a special instance of Theorem \ref{HGB}
\begin{equation}\label{simplified-GM}
\int_{\mathbb{S}_{\heis}}\cK_0 \, d\mathcal{H}^{3}_{cc} =0.
\end{equation}
Equation \ref{simplified-GM} can also be verified directly using \eqref{eq:K0Kor}.
\end{example}

Further examples can be found in the appendix.

\section{Application: a simplified Steiner's formula}\label{examples}

The main application of the Gauss--Bonnet Theorem concerns
a simplification to the Steiner's formula recently proved in \cite{BFFVW} {\it cc} neighborhoods of
surfaces without characteristic points. Let us first recall some notation and background from \cite{BFFVW}.

Let $\Omega \subset \heis$ be an open, bounded and regular domain with
smooth boundary $\bound$. Let $\del_{cc}$ be the signed $cc$-distance function from $\bound$.
The $\ep$-neighborhood $\Omega_{\ep}$ of $\Omega$ with respect to
the $cc$-metric is given by
\begin{equation}\label{eq:neicha}
\Omega_{\ep}:= \Omega \cup \left\{ g \in \heis : 0 \leq \del_{cc}(g) < \ep\right\}.
\end{equation}

We define the \textit{iterated divergences} of $\del_{cc}$ as follows:
\begin{equation*}
\divh^{(i)}(\del_{cc}) =\left\{\begin{array}{l}
 1, \quad \quad \quad \quad \quad \quad \quad \quad \quad \quad \mbox{$i=0$} \\
\divh(\divh^{(i-1)}\del_{cc} \cdotp \hgrad \del_{cc}), \quad \quad \mbox{$i \geq 1$}
            \end{array}\right.
\end{equation*}
Finally we put
\begin{equation}\begin{split}\label{ABCDE}
&A:= \hlapl \del_{cc}, \quad \quad B:=- (X_{3}\del_{cc})^{2}, \quad \quad C:=(X_{1}\del_{cc})(X_{32}\del_{cc})-(X_{2}\del_{cc})(X_{31}\del_{cc}), \\
&D:= X_{33}\del_{cc},\quad \quad E:=(X_{31}\del_{cc})^{2}+(X_{32}\del_{cc})^{2}.
\end{split}\end{equation}

In order to make the paper self-contained, we recall a technical result from \cite{BFFVW}.
Define the operator $g$ acting on smooth real valued functions as
$$
g(\alpha):= \langle \hgrad \alpha, \hgrad \del \rangle.
$$
For a real-valued function $h$ we have
\begin{equation}\label{div-of-h-formula}
\divh(h\,\nabla_H\delta_{cc}) = h \, \hlapl \del_{cc} + \scal{\hgrad h}{\hgrad \del_{cc}} = h \, A + g(h).
\end{equation}
Note that $g$ is linear and satisfies the Leibniz rule, i.e.
\begin{align*}
&g(\alpha + \beta)  = g(\alpha)+g(\beta), \\
&g(\alpha \, \beta)  = g(\alpha) \beta + \alpha g(\beta).
\end{align*}

The following lemma holds (see \cite{BFFVW}, Lemma 4.2). The proof involves a number of lengthy calculations using higher derivatives of the {\it cc} distance function.

\begin{lemma}\label{itdiv}
The following relations hold:
\begin{align}
g(1)&= 0 \label{eq:g1},\\
g(A)&= B+2C -A^{2}\label{eq:gA},\\
g(B)&= 0 \label{eq:gB},\\
g(C)&= D - AC \label{eq:gC},\\
g(D)&= -E \label{eq:gD},\\
g(E)&= -2AE+2CD \label{eq:gE}.
\end{align}
\end{lemma}

We now state the main result of this section, which gives a simplification to the main theorem of \cite{BFFVW} (a Steiner's formula for the {\it cc} distance function).

\begin{theorem}\label{Torus}
Let $\Omega \subset \heis$ be an open, bounded and regular domain whose boundary
$\bound$ is a Euclidean $C^{2}$-smooth compact and oriented surface
 with no characteristic points. For sufficiently small $\ep >0$ define the $\ep$-neighborhood $\Omega_{\ep}$
 of $\Omega$ with respect to the $cc$-metric as in (\ref{eq:neicha}). Then
 \begin{equation*}
 \begin{aligned}
\leb^{3}(\Omega_{\ep})&= \leb^{3}(\Omega) + \ep \mathcal{H}_{cc}^{3}(\bound) +  \dfrac{\ep^{2}}{2}\int_{\bound} A \, d\mathcal{H}_{cc}^{3} +
\dfrac{\ep^{3}}{3!}\int_{\bound}C\, d\mathcal{H}_{cc}^{3} +\\
&\sum_{j=1}^{+\infty}\left( \int_{\bound} (B^{j-1}D) \, d\mathcal{H}_{cc}^{3}\right) \dfrac{\ep^{2j+2}}{(2j+2)!} +
\sum_{j=1}^{+\infty}\left( \int_{\bound} (B^{j-1}(AD-E)) \, d\mathcal{H}_{cc}^{3}\right) \dfrac{\ep^{2j+3}}{(2j+3)!},
\end{aligned} \end{equation*}
where $A$, $B$, $C$, $D$ and $E$ are given as in \eqref{ABCDE}.
\end{theorem}

\begin{proof}
Let us write $\partial_{f} \Omega_{\ep} := \{ g\in \heis: \del_{cc}(g)=\ep \}$.
The evolution of the non-characteristic set $\bound$ can be explicitly described if
we know the defining function of $\bound$, see \cite{AF07}. In our situation, we
can assume that $\bound = \{ g \in \heis: \del_{cc}(g)=0\}$. In particular, the results
from \cite{AF07} tell us that there exists a map $\Mn:[0,\ep]\times \bound \to \reals^{3}$,
such that $\Mn(\cdot,g):[0,\ep]\to \reals^{3}$ is continuous, $\Mn(\ep,\cdot)\to \partial_{f} \Omega_{\ep}$
is smooth and $\Mn(0,g)= \mbox{id}|_{\bound}(g)$. We claim that there exists a continuous map $a(\ep,g):= \textrm{angle}(T_{\Mn(\ep,g)}\partial_{f} \Omega_{\ep}; H_{\Mn(\ep,g)})$.
Since $\bound$ has no characteristic points, $a(0,g)>0$ for every $g \in \bound$. The continuity of
$a(\cdot,\cdot)$ implies that there exists $\ep_0$, $0<\ep_{0}<\ep$, so that
$a(s,g) >0$ for every $g \in \bound$ and for every $s \in (0,\ep_0)$. In particular this shows that
we can choose a sufficiently small $\ep>0$ so that $\partial_{f} \Omega_{\ep}$ is still a non-characteristic set.

We will use the following proposition from \cite{BFFVW} (see Proposition 3.3) which can be proved with the help of the sub-Riemannian divergence formula.

\begin{proposition}\label{prop-3-3}
Let $h:\Omega_{\ep} \to \reals$ be a $C^1$ function. Then the vector field $h \, \nabla_H\delta_{cc}:\Omega_{\ep} \to \reals^3$ satisfies
$$
\int_{\{ s < \delta_{cc} < t \}} \divh(h \, \nabla_H\delta_{cc}) \, d{\mathcal L}^3 = \int_{\delta_{cc}^{-1}(t)} h \, d{\mathcal H}^3_{cc} - \int_{\delta_{cc}^{-1}(s)} h \, d{\mathcal H}^3_{cc}.
$$
\end{proposition}

We are interested in a Taylor series expansion of the function
$$
\ep \mapsto \leb^{3}(\Omega_{\ep}),
$$
about $\ep = 0$. The analyticity of this function has been already proved in \cite{BFFVW}, therefore
let us denote by $f^{(i)}(\ep)$ the $i^{th}$ derivative of $\ep \mapsto \leb^{3}(\Omega_{\ep})$.
The first three elements of the expansion are obtained as in \cite{BFFVW}. For the other terms,
we need to recall that by Theorem 3.4 of \cite{BFFVW}, for every $i \geq 0$,
\begin{equation}\label{eq:lemma}
f^{(i)}(s)= \int_{\del_{cc}^{-1}(s)}(\divh^{(i-1)}\del_{cc})\, d\mathcal{H}_{cc}^{3}, \quad \textrm{ for every } s \in [0,\ep_0).
\end{equation}
In particular, for $i=3$, $\divh^{(2)}\del_{cc}= B+2C$ and the expression $B+C$ coincides with the horizontal
Gaussian curvature $\cK_0$. Therefore we can apply the Gauss--Bonnet Theorem \ref{HGB} to obtain
$$
f^{(3)}(s)= \int_{\del_{cc}^{-1}(s)}C\, d\mathcal{H}_{cc}^{3}
$$
and hence
$$
f^{(3)}(0)= \int_{\bound}C\, d\mathcal{H}_{cc}^{3}.
$$
We now claim that
\begin{equation}\label{2j2}
f^{(2j+2)}(s) = \int_{\del_{cc}^{-1}(s)} (B^{j-1}D) \, d{\mathcal H}_{cc}^3 \qquad \mbox{for $j \ge 1$}
\end{equation}
and
\begin{equation}\label{2j3}
f^{(2j+3)}(s) = \int_{\del_{cc}^{-1}(s)} (B^{j-1}(AD-E)) \, d{\mathcal H}_{cc}^3 \qquad \mbox{for $j \ge 1$,}
\end{equation}
from which the indicated values of the coefficients in the series expansion of $\ep \mapsto {\mathcal L}^3(\Omega_\ep)$ are obtained by setting $s=0$.

The formulas in \eqref{2j2} and \eqref{2j3} can be obtained inductively by evaluating difference quotient approximations to the indicated derivatives using the inductive hypothesis and the divergence formula in Proposition \ref{prop-3-3}. First,
\begin{equation*}\begin{split}
f^{(4)}(s) &= \lim_{\ep \to 0} \tfrac1\ep \left( f^{(3)}(s+\ep) - f^{(3)}(s) \right) \\
&= \lim_{\ep \to 0} \tfrac1\ep \left( \int_{\del_{cc}^{-1}(s+\ep)} C \, d\mathcal{H}_{cc}^3 - \int_{\del_{cc}^{-1}(s)} C \, d\mathcal{H}_{cc}^3 \right) \\
&= \lim_{\ep \to 0} \tfrac1\ep \int_{\{ s < \del_{cc} < s+\ep \}} \divh(C \, \nabla_H \del_{cc}) \, d{\mathcal L}^3
= \int_{\del_{cc}^{-1}(s)} \divh(C \, \nabla_H \del_{cc}) \, d{\mathcal H}^3_{cc}.
\end{split}\end{equation*}
By \eqref{div-of-h-formula} and Lemma \ref{itdiv}, $\divh(C\,\nabla_H \del_{cc}) = CA + g(C) = AC + (D-AC) = D$ and so \eqref{2j2} holds in the case $j=1$.

Similarly, for $j \ge 2$,
\begin{equation*}\begin{split}
f^{(2j+2)}(s) &= \lim_{\ep \to 0} \tfrac1\ep \left( f^{(2j+1)}(s+\ep) - f^{(2j+1)}(s) \right) \\
&= \lim_{\ep \to 0} \tfrac1\ep \left( \int_{\del_{cc}^{-1}(s+\ep)} (B^{j-2}(AD-E)) \, d\mathcal{H}_{cc}^3 - \int_{\del_{cc}^{-1}(s)} (B^{j-2}(AD-E)) \, d\mathcal{H}_{cc}^3 \right) \\
&= \lim_{\ep \to 0} \tfrac1\ep \int_{\{ s < \del_{cc} < s+\ep \}} \divh(B^{j-2}(AD-E) \, \nabla_H \del_{cc}) \, d{\mathcal L}^3 \\
&= \int_{\del_{cc}^{-1}(s)} \divh(B^{j-2}(AD-E) \, \nabla_H \del_{cc}) \, d{\mathcal H}^3_{cc}.
\end{split}\end{equation*}
By \eqref{div-of-h-formula} and Lemma \ref{itdiv},
\begin{equation*}\begin{split}
\divh(B^{j-2}(AD-E)\,\nabla_H \del_{cc})
&= B^{j-2}(AD-E)A + g(B^{j-2}(AD-E)) \\
&= B^{j-2}(A^2D-AE) + B^{j-2}(Ag(D)+Dg(A)-g(E)) \\
&= B^{j-2}(A^2D - AE - AE + BD + 2CD - A^2D +2AE - 2CD) \\
&= B^{j-1}D
\end{split}\end{equation*}
and so \eqref{2j2} holds in the case $j\ge 2$.

A similar computation establishes \eqref{2j3} for all $j \ge 1$. This completes the proof.
\end{proof}

\section{Questions and remarks}\label{questions}

As is clear from the proof of the Gauss--Bonnet theorem,
it is of crucial interest to prove local summability of the sub-Riemannian
Gaussian curvature $\cK_0$ around isolated characteristic points, with respect to the Heisenberg
perimeter measure.
For the horizontal mean curvature $\cH_{0}$ of $\Sigma$
this is an established result, see \cite{DGN12}.
In the same work \cite{DGN12}, it is showed that the
situation could change dramatically if we address the problem of local integrability
of $\cH_0$ with respect to the Riemannian surface measure,
near the characteristic set.
In this case, it is conjectured that we should have
locally integrability if we are close to an
\textit{isolated characteristic point} of $\Sigma$, and it is presented
a counterexample in the case in which
the characteristic set $\mathrm{char}(\Sigma)$ is 1-dimensional.

\begin{question}
Is the sub-Riemannian Gaussian
curvature $\cK_0$ locally summable with respect to the
Heisenberg perimeter measure, near the characteristic set?
\end{question}

Recalling \eqref{eq:summ},
it is clear that the local summability of $\cK_0$ is
closely related to the integrability of $\cH_0$
near isolated characteristic points with respect to the
Riemannian surface measure. As far as we know, the best
results in this direction are those of \cite{DGN12},
which provide a class of examples where we have local
integrability. In the same spirit of \cite{DGN12},
we have the following result.

\begin{proposition}
Let $\Sigma \subset \heis$ be a Euclidean $C^{2}$-smooth surface.
Suppose $\Sigma$ has cylindrical symmetry near an isolated characteristic point $g$,
then $\cK_0 \in L^{1}(\Sigma, d\mathcal{H}^{3}_{cc})$.
\end{proposition}

 \begin{proof}
  Without loss of generality we can assume that the isolated
	characteristic point is the origin $0=(0,0,0)$, and
  that locally around $0$ the surface $\Sigma$ is given by the
	0-level set of the function
  $$u(x_1,x_2,x_3):=x_3 - f\left(\dfrac{x_{1}^{2}+x_{2}^{2}}{4}\right),$$
   where $f\in C^{2}$. For simplicity, let us denote by
	$f(r):=f\left(\tfrac{x_{1}^{2}+x_{2}^{2}}{4}\right).$
	 Then,
  $$X_3 u = 1, \quad X_1 u = -\tfrac{1}{2}(x_2 +x_1 f'(r)) \quad \textrm{and} \quad X_2 u = \tfrac{1}{2}(x_1 - x_2 f'(r)),$$
  therefore $\|\hgrad u\|_{\he} = \tfrac{1}{2} \sqrt{x_{1}^{2}+x_{2}^{2}} \sqrt{1+(f'(r))^{2}}$.
  In order to compute the sub-Riemannian Gaussian curvature $\cK_0$ we need
  \begin{equation*}
   \begin{aligned}
    X_1 \left( \dfrac{1}{\|\hgrad u\|_{\he}}\right)&= 2 \partial_{x_1}\left( (x_{1}^{2}+x_{2}^{2})^{-\tfrac{1}{2}}(1+(f')^{2})^{-\tfrac{1}{2}}\right)\\
    &=- \dfrac{2 x_1 (1+(f')^{2})+f' \, f'' \, x_1 (x_{1}^{2}+x_{2}^{2})}{(x_{1}^{2}+x_{2}^{2})^{3/2}(1+(f')^{2})^{3/2}},
   \end{aligned}  \end{equation*}
 and
 \begin{equation*}
   \begin{aligned}
    X_2 \left( \dfrac{1}{\|\hgrad u\|_{\he}}\right)&= 2 \partial_{x_2}\left( (x_{1}^{2}+x_{2}^{2})^{-\tfrac{1}{2}}(1+(f')^{2})^{-\tfrac{1}{2}}\right)\\
    &=- \dfrac{2 x_2 (1+(f')^{2})+f' \, f'' \, x_2 (x_{1}^{2}+x_{2}^{2})}{(x_{1}^{2}+x_{2}^{2})^{3/2}(1+(f')^{2})^{3/2}}.
   \end{aligned}  \end{equation*}
After some simplifications we get
\begin{equation}\label{cK-formula}
\cK_0 = -\dfrac{2}{(x_{1}^{2}+x_{2}^{2})(1+(f')^{2})} + \dfrac{ f'\, f''}{(1+(f')^{2})^{2}},
\end{equation}
which is summable.
 \end{proof}

A generalization of \eqref{cK-formula} appears in Example \ref{example-x3graph}. \\

It is obvious that the local summability of $\|\hgrad u\|_{\he}^{-1}$
implies the local summability of the sub-Riemannian Gaussian curvature $\cK_0$, but it is not necessary,
as showed by the following example.

\begin{example}
 Let $\Sigma = \{ (x_1,x_2,x_3)\in \heis :u(x_1,x_2,x_3)=0 \},$
for $$u(x_1,x_2,x_3) = x_3 - \dfrac{x_1\, x_2}{2} -\exp\left(-(x_{1}^{2}+x_{2}^{2})^{-2} \right).$$
  The origin $0=(0,0,0)$ is an
isolated characteristic point of $\Sigma$, indeed
 $$X_{1}u = -x_2 -\dfrac{4 x_1 \exp\left(-(x_{1}^{2}+x_{2}^{2})^{-2}\right)}{(x_{1}^{2}+x_{2}^{2})^{3}}, \quad \textrm{and} \quad
 X_{2}u = -\dfrac{4 x_2 \exp\left(-(x_{1}^{2}+x_{2}^{2})^{-2}\right)}{(x_{1}^{2}+x_{2}^{2})^{3}}.$$
 Switching to polar coordinates $x=r \cos(\theta), y=r \sin(\theta),$ we have
 $$\|\hgrad u\|_{\he}^{2} = r^{2} \sin^{2}(\theta) + \dfrac{16 \exp(-2r^{-4})}{r^{10}} + \dfrac{4r^{2}\sin(2\theta)\exp(-r^{-4})}{r^{6}} \leq r^{2}\sin^{2}(\theta) + \exp(-r^{-4}).$$
Therefore we are interested in the summability of the following integral
\begin{equation}\label{eq:inte}
\int_{0}^{\ep} \int_{0}^{2\pi}\dfrac{r}{\sqrt{r^{2}\sin^{2}(\theta)+ \exp(-r^{-4})}} d\theta \, dr .
\end{equation}
Now, setting $g(r):= r^{-1} \exp(-r^{-2}),$ we have
\begin{equation*}
\begin{aligned}
\int_{0}^{\ep} &\int_{0}^{2\pi}\dfrac{r}{\sqrt{r^{2}\sin^{2}(\theta)+ \exp(-r^{-4})}} d\theta \, dr \gtrsim \int_{0}^{\ep}\int_{0}^{2\pi}\dfrac{1}{|\sin(\theta)|+ g(r)}d\theta \, dr\\
&\geq \int_{0}^{\ep}\left( \int_{0}^{\delta} \dfrac{1}{|\sin(\theta)|+g(r)}d\theta\right)dr \approx \int_{0}^{\ep}\left(\int_{0}^{\delta}\dfrac{1}{\theta+g(r)}d\theta\right)dr\\
&=\int_{0}^{\ep}\ln\left(1+\dfrac{\delta}{g(r)}\right)dr,
\end{aligned} \end{equation*}
which is divergent.

Unfortunately this does not provide an example
 of a surface with isolated characteristic points
 whose sub-Riemannian Gaussian curvature $\cK_0$ is \textit{not}
locally integrable with respect to
 the Heisenberg perimeter measure. Indeed,
after quite long computations, one can prove that
in this case the sub-Riemannian Gaussian curvature
$\cK_0$ is actually locally integrable.
\end{example}

Our second question relates to the possible connection existing between
the sub-Riemannian Gaussian curvature $\cK_0$ and one of the integrands appearing
in the localized Steiner's formula proved in \cite{BFFVW}.
In particular, if we consider a Euclidean $C^{2}$-smooth and regular surface $\Sigma \subset \heis$ a defining function $u$ such that $\|\hgrad u\|_{\he}=1$, then we have that
$$\cK_0 = B + C,$$
where $B$ and $C$ are defined as in Section \ref{examples}.

\begin{question}
Is there any explanation for
the discrepancy between the horizontal
Gaussian curvature $\cK_0$ and the fourth integrand appearing in the
Steiner's formula proved in \cite{BFFVW}?
\end{question}

We want to mention that the expression we got for the sub-Riemannian Gaussian curvature
$\cK_0$ appears also in the paper \cite{DGN07} in the study of stability properties of minimal surfaces in $\heis$, and also in the upcoming manuscript \cite{ChMT}.

We end this paper with a Fenchel-type theorem for fully horizontal curves.

\begin{theorem}\label{Fenchel}
Let $\gamma:[a,b]\to \heis$ be a Euclidean
$C^{2}$-smooth, regular, closed and fully horizontal curve.
Then
 \begin{equation}
\int_{\gamma} k_{\gamma}^{0}\, d\cu_{\he} > 2\pi,
 \end{equation}
 where $d\cu_{\he}$ is the standard Heisenberg
length measure of Definition \ref{length}.
\end{theorem}

\begin{proof}
Define the projected curve
$$
\tilde{\gamma}(t):= \Pi(\gamma)(t) = (\gamma_1(t), \gamma_2(t), 0)_{\{e_1,e_2,e_3\}}.
$$
Then $\tilde{\gamma}$ is a planar Euclidean $C^{2}$-smooth, regular and closed curve
whose curvature $k$ coincides precisely with the sub-Riemannian curvature $k_{\gamma}^{0}$ of $\gamma$.
Due to the fact that the curve is horizontal,
$$
\int_{\gamma} k_{\gamma}^{0}\, d\cu_{\he} = \int_{\tilde{\gamma}} k \, d\cu_{\reals^{2}},
$$
where $d\cu_{\reals^{2}}$ denotes the standard Euclidean length measure in $\reals^{2}$.

The classical Fenchel Theorem (see \cite{dC76}, Chapter 5.7) assures that
$$
\int_{\tilde{\gamma}} k \, d\cu_{\reals^{2}} \geq 2\pi,
$$
and states that equality is achieved if and only if the curve $\tilde{\gamma}$ is convex. It is a well known fact concerning horizontal curves that the projection of a closed fully horizontal curve $\gamma$ has enclosed oriented area equal to 0.
Therefore its projection $\tilde{\gamma}$ cannot be a convex curve, and the inequality has to be strict. This completes the proof.
\end{proof}

\appendix
\section{Examples}
We want to collect here a list of explicit examples where
we compute the sub-Riemannian Gaussian curvature explicitly.
\begin{example}
Any vertically ruled surface Euclidean $C^{2}$-smooth surface
$\Sigma$ has vanishing sub-Riemannian Gaussian curvature, i.e.
if
$$\Sigma = \{ (x_1, x_2, x_3)\in \heis: f(x_1,x_2)=0\},$$
for $f \in C^{2}(\mathbb{R}^2)$, then $\cK_0 =0$.
In particular, every vertical plane has constant
sub-Riemannian Gaussian curvature $\cK_0 = 0$.
\end{example}

\begin{example}
The horizontal plane through the origin,
$\Sigma = \{ (x_1,x_2,x_3) \in \heis: x_3=0\},$
has
$$\cK_0 = -\dfrac{2}{(x_{1}^{2} + x_{2}^{2})}.$$
\end{example}

\begin{example}
The Kor\'{a}nyi sphere,
$\Sigma = \{ (x_1,x_2,x_3) \in \heis : (x_{1}^2 + x_{2}^2)^2 + 16 x_{3}^2 -1 =0 \},$
has
$$\cK_0 = -\dfrac{2}{(x_{1}^2 + x_{2}^2)} + 6 (x_{1}^2 + x_{2}^2).$$
\end{example}

\begin{example}
Let $\alpha >0$. The paraboloid,
$\Sigma = \{ (x_1, x_2, x_3)\in \heis: x_3 = \alpha (x_{1}^2 + x_{2}^2)\},$
has $$\cK_0 = - \dfrac{1}{1+ 16 \alpha^2} \, \dfrac{2}{(x_{1}^2 + x_{2}^2)}.$$
\end{example}

\begin{example}\label{example-x3graph}
Every surface $\Sigma$ given as a $x_3$-graph, $\Sigma = \{ (x_1,x_2,x_3)\in \heis: x_3 = f(x_1,x_2)\}$, with $f\in C^{2}(\mathbb{R}^2)$, has
$$
\cK_0 = -\dfrac{2}{\|\hgrad u \|_{\mathnormal{H}}^{2}}+
\dfrac{1}{\|\hgrad u \|_{\mathnormal{H}}^{4}} (\mathrm{Hess}f) \left(\hgrad u, J \hgrad u\right),
$$
where $u(x_1, x_2, x_3):= x_3 - f(x_1,x_2)$.
\end{example}

For $x_3$-graphs, we have another useful result that
provides a sufficient condition for the sub-Riemannian
Gaussian curvature $\cK_0$ to vanish.

\begin{lemma}
Let $\Sigma \subset \heis$ be as before. Let $g \in \Sigma$ and suppose that
$\Sigma$ is a Euclidean $C^{2}$-smooth $x_3$-graph and $X_1u$ and $X_2u$ are linearly dependent
in a neighborhood of $g$. Then $\cK_0(g)=0$.
\end{lemma}

\begin{proof}
In the case of a $x_3$-graph, $X_3 u =1$ and we have
$$
\cK_0 = -\frac1{|\nabla_0 u|^2} + \frac{X_2 u X_1(\tfrac12|\nabla_0 u|^2) - X_1 u X_2(\tfrac12|\nabla_0 u|^2)}{|\nabla_0 u|^4}.
$$
Assume that $aX_1 u + bX_2 u = $ in a neighborhood of $g$. Let us suppose that $b \ne 0$; the case $a \ne 0$ is similar. Without loss of generality, assume that $b=1$. We expand
\begin{equation}\label{x3graph}
\cK_0 = \frac{-(X_1u)^2-(X_2u)^2 + X_1uX_2uX_1X_1u+(X_2u)^2X_1X_2u-(X_1u)^2X_2X_1u-X_1uX_2uX_2X_2u}{|\nabla_0 u|^4}
\end{equation}
and use the identities $X_2u=aX_1u$,
$$
X_1X_2u=aX_1X_1u,
$$
$$
X_2X_1u=X_1X_2u-X_3u=aX_1X_1u-1
$$
and
$$
X_2X_2u=aX_2X_1u=a-a^2X_1X_1u
$$
to rewrite the numerator of \eqref{x3graph} entirely in terms of $X_1u$ and $X_1X_1u$. A straightforward computation shows that the expression for $\cK_0$ vanishes. The case when $a\ne 0$ is similar.
\end{proof}

\begin{example}
Every surface $\Sigma$ given as a $x_1$-graph,
$\Sigma = \{ (x_1,x_2,x_3)\in \heis: x_1 = f(x_2,x_3) \}$
with $f\in C^{2}(\mathbb{R}^2)$, has
$$\cK_0 = -\dfrac{f_3^2}{\|\hgrad u \|_{\mathnormal{H}}^2}
+ \dfrac{(x_1^2-x_2^2) f_{33}}{8\, \|\hgrad u \|_{\mathnormal{H}}^4}
\left( x_1 \, f_3 + \dfrac{x_1 x_2 \, f_3^2}{2}\right)
-\dfrac{f_{23} (1+ \tfrac{x_2}{2}f_3)}{\|\hgrad u \|_{\mathnormal{H}}^2}
+ \dfrac{1}{\|\hgrad u \|_{\mathnormal{H}}^4}\left( \dfrac{x_1^2 \, f_3^3}{8} + \dfrac{(1+\tfrac{x_2}{2}f_3)}{2}f_3\right),$$
where $u(x_1, x_2, x_3):= x_1 - f(x_2,x_3)$. A similar result
holds for $x_2$-graphs.
\end{example}

\bibliography{GB}
\bibliographystyle{acm}
\nocite{*}
\end{document}